\documentclass{amsart}
\usepackage{amsmath,amsthm,amssymb,amsfonts}
\usepackage[a4paper]{geometry}
\usepackage{a4wide}
\usepackage{graphicx}              
\usepackage{mathpazo}
\usepackage[T1]{fontenc}
\usepackage[utf8]{inputenc}
\usepackage{enumitem}
\usepackage{hyperref}
\usepackage{color}
\usepackage{xcolor}
\hypersetup{
    colorlinks,
    linkcolor={red!80!black},
    citecolor={blue!80!black},
    urlcolor={blue!80!black}
}
\usepackage{todonotes}

\newcommand{\jjhidtodo}[1]{}

\newcommand{\niy}[1]{\todo[inline,color=yellow!40,bordercolor=orange]{\texttt{Not implemented yet}}}

\usepackage{booktabs}
\usepackage{ mathrsfs } 
\usetikzlibrary{matrix,arrows, cd, calc}  
\usetikzlibrary{graphs}    

\usepackage[margin=10pt,font=normalsize,labelfont=bf, labelsep=endash]{caption}

\makeatletter
\def\thm@space@setup{\thm@preskip=4pt
\thm@postskip=2pt}
\makeatother

\numberwithin{equation}{section}
\newtheorem{theorem}[equation]{Theorem}

\newtheorem{corollary}[equation]{Corollary}

\newtheorem{proposition}[equation]{Proposition}

\newtheorem{lemma}[equation]{Lemma}

\theoremstyle{definition}

\newtheorem{definition}[equation]{Definition}
\newtheorem{example}[equation]{Example}
\newtheorem{remark}[equation]{Remark}

\newcommand{\cP}{{\mathcal P}}

\newcommand{\cN}{{\mathcal N}}
\newcommand{\cM}{{\mathcal M}}

\newcommand{\goodquotient}{\mathbin{
  \mathchoice{\left/\mkern-6mu\right/}
    {/\mkern-5mu/}
    {/\mkern-5mu/}
    {/\mkern-5mu/}}}

\newcommand{\Hom}{\mathrm{Hom}}

\newcommand{\Spec}{\mathrm{Spec}\,}

\newcommand{\Sym}{\textrm{Sym}}
\newcommand{\Mor}{\textrm{Mor}}

\newcommand{\End}{\textrm{End}}

\newcommand{\Set}{\bd{Set}}

\newcommand{\Alg}{\bd{Alg}}

\newcommand{\Mod}{\bd{Mod}}
\newcommand{\Rep}{\bd{Rep}}
\DeclareMathOperator{\Qcoh}{\operatorname{Qcoh}}%
\DeclareMathOperator{\pr}{pr}
\DeclareMathOperator{\GL}{GL}%
\DeclareMathOperator{\SL}{SL}%

\newcommand{\lambdachi}{\lambda\mbox{-}\Chi}

\newcommand{\bd}[1]{\mathbf{#1}}  

\newcommand{\mm}{\mathfrak{m}}%

\newcommand{\Group}{\mathbf{G}}%

\newcommand{\Gbar}{\overline{\mathbf{G}}}%
\newcommand{\Tbar}{\overline{\mathbf{T}}}%

\newcommand{\Gmult}{\mathbb{G}_m}%
\newcommand{\Gmultkkbar}{\mathbb{G}_{m,\kkbar}}%
\newcommand{\kk}{k}%
\newcommand{\kkbar}{\overline{k}}%
\newcommand{\varX}{X}%
\newcommand{\Xplus}{\varX^+}%
\newcommand{\ione}[1]{i_{#1}}%
\newcommand{\ioneX}{\ione{\varX}}
\newcommand{\iinfty}[1]{\pi_{#1}}%
\newcommand{\iinftyX}{\iinfty{\varX}}%
\newcommand{\isection}[1]{s_{#1}}%
\newcommand{\isectionX}{\isection{\varX}}%
\newcommand{\BBname}{Bia{\l}ynicki-Birula}%
\DeclareMathOperator{\Map}{Map}

\DeclareMathOperator{\colim}{colim}

\DeclareMathOperator{\id}{id}%
\newcommand{\into}{\hookrightarrow}%
\newcommand{\onto}{\twoheadrightarrow}%
\newcommand{\mubar}{\overline{\mu}}%

\newcommand{\RepG}{\Rep_{\Group}}%
\newcommand{\RepGbar}{\Rep_{\Gbar}}%
\newcommand{\RepGlambda}{\Rep_{\Group}[\lambda]}%
\newcommand{\QcohG}[1]{\Qcoh_{\Group}(#1)}%
\newcommand{\QcohGft}[1]{\Qcoh_{\Group}^{\mathrm{ft}}(#1)}%
\newcommand{\QcohGX}{\QcohG{\varX}}%
\newcommand{\OO}{\mathcal{O}}%
\newcommand{\tensor}{\otimes}%
\renewcommand{\AA}{\mathcal{A}}%
\newcommand{\MM}{\mathcal{M}}%
\newcommand{\Chi}{\mathcal{X}}%
\newcommand{\Xhat}{\hat{X}}%

\begin{document}

\title{Bia{\l}ynicki-Birula decomposition for reductive groups in positive characteristic}
\author{Joachim Jelisiejew}
\thanks{Institute of Mathematics, University of Warsaw, \url{jjelisiejew@mimuw.edu.pl}, supported by Polish National Science
Center, project 2017/26/D/ST1/00755 and Institute of Mathematics, Polish
Academy of Sciences.}

    \author{{\L{}}ukasz Sienkiewicz}
    \thanks{Institute of Mathematics, University of Warsaw,
        \url{lusiek@mimuw.edu.pl}}

\begin{abstract}
    We prove the existence of \BBname{} decomposition for \emph{Kempf
    monoids}, which form a large class that contains all monoids with
    reductive unit group in all characteristics. This extends the existence
    statements from~\cite{jelisiejew_sienkiewicz__BB, AHR}.
\end{abstract}

\maketitle

\section{Introduction}

    Let $\kk$ be a field.
    By a \emph{linear group} $\Group$ we mean a smooth affine group scheme of finite
    type over $\kk$. When talking about a $\Group$-action we always mean a left
    $\Group$-action.
    Let $\Group$ be a linear group and $\Gbar$ be a geometrically integral affine algebraic monoid with
    zero and with unit group $\Group$.
    For a fixed $\Group$-scheme $X$ over $\kk$, its \BBname{} decomposition is a
    set-valued functor $\Xplus := \Map(\Gbar, X)^{\Group}$. Explicitly, for a
    $\kk$-scheme $S$ which we endow with the trivial $\Group$-action, the functor
    is
    given by
    \[
        \Xplus(S) := \left\{ \varphi\colon \Gbar \times S \to X\ |\
            \varphi\mbox{ is }\Group\mbox{-equivariant}\right\}.
    \]
    Intuitively, the functor $\Xplus$ parameterizes $\Group$-orbits in $X$ which
    compactify to $\Gbar$-orbits. Indeed, its $\kk$-points are
    $\Group$-equivariant morphisms $f\colon\Gbar \to X$.
    The restrictions $f\mapsto f(1)\in X$ and $f\mapsto f(0)\in X^{\Group}$
    give maps $\Xplus(\kk)\to X(\kk)$ and $\Xplus(\kk)\to X^{\Group}(\kk)$.
    The map $f$ is a partial compactification of the $\Group$-orbit of $f(1)$ and
    $f(0)$ is a ``limit'' or the most degenerate point of this compactification.
    The evaluations at zero and one extend to maps of functors:
    \begin{itemize}
        \item the map $\ioneX\colon\Xplus \to X$ that sends $\varphi$ to $\varphi_{|1 \times S} \colon S\to X$.
        \item the \emph{limit map} $\iinftyX\colon \Xplus \to X^{\Group}$ that
            sends $\varphi$ to $\varphi_{|0\times S}\colon S\to X^{\Group}$. The
            limit map has a section $\isectionX$ that sends
            $\varphi_0\colon S \to X^{\Group}$ to the constant family $\varphi
            = \varphi_0 \circ pr_2\colon \Gbar \times S\to X$.
    \end{itemize}
    Altogether, we obtain the diagram
    $\begin{tikzcd}
            X^{\Group}\arrow[r, bend left, "\isectionX"] & \Xplus \arrow[r, "\ioneX"]\arrow[l, "\iinftyX"] & X
        \end{tikzcd}$.
    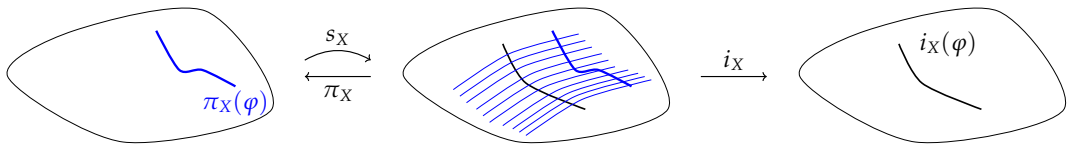
\begin{figure}[h]
        \[\resizebox{0.95\textwidth}{!}{
                \begin{tikzpicture}
                    \draw plot [smooth cycle] coordinates {(0, 0) (2, 1.4) (6,
                        2) (8,
                    -1) (3.5, -2)};
                    \draw[line width=0.5mm] plot [smooth] coordinates
                    {(3, 1) (3.7, -0.2) (5.5, -1)};
                    \draw[blue] plot [smooth] coordinates {(5.3,
                    1.3) (3.2, 0.7) (1.5, -0.4)};
                    \begin{scope}[shift={(0.2, -0.2)}]
                        \draw[blue] plot [smooth] coordinates {(5.3,
                        1.3) (3.2, 0.7) (1.5, -0.4)};
                    \end{scope};
                    \begin{scope}[shift={(0.4, -0.4)}]
                        \draw[blue] plot [smooth] coordinates {(5.3,
                        1.3) (3.2, 0.7) (1.5, -0.4)};
                    \end{scope};
                    \begin{scope}[shift={(0.7, -0.6)}]
                        \draw[blue] plot [smooth] coordinates {(5.3,
                        1.3) (3.2, 0.7) (1.5, -0.4)};
                    \end{scope};
                    \begin{scope}[shift={(0.9, -0.8)}]
                        \draw[blue] plot [smooth] coordinates {(5.3,
                        1.3) (3.2, 0.7) (1.5, -0.4)};
                    \end{scope};
                    \begin{scope}[shift={(1.1, -1)}]
                        \draw[blue] plot [smooth] coordinates {(5.3,
                        1.3) (3.2, 0.7) (1.5, -0.4)};
                    \end{scope};
                    \begin{scope}[shift={(1.5, -1.1)}]
                        \draw[blue] plot [smooth] coordinates {(5.3,
                        1.3) (3.2, 0.7) (1.5, -0.4)};
                    \end{scope};
                    \begin{scope}[shift={(1.8, -1.2)}]
                        \draw[blue] plot [smooth] coordinates {(5.3,
                        1.3) (3.2, 0.7) (1.5, -0.4)};
                    \end{scope};
                    \begin{scope}[shift={(2, -1.3)}]
                        \draw[blue] plot [smooth] coordinates {(5.3,
                        1.3) (3.2, 0.7) (1.5, -0.4)};
                    \end{scope};
                    \begin{scope}[shift={(2.2, -1.4)}]
                        \draw[blue] plot [smooth] coordinates {(5.3,
                        1.3) (3.2, 0.7) (1.5, -0.4)};
                    \end{scope};

                    \draw[blue, line width=0.7mm] plot [smooth] coordinates
                    {(4.5, 1.4) (5.2, 0.2) (5.9, 0.2) (6.9, -0.3)};

                    \path[->,line width=1pt] (9,0) edge (11,0);
                    \node (v1) at (10, 0.5) {{\Huge $\ione{\varX}$}};
                    \begin{scope}[shift={(12, 0)}]
                        \draw plot [smooth cycle] coordinates {(0, 0) (2, 1.4) (6,
                            2) (8,
                        -1) (3.5, -2)};
                        \draw[line width=0.5mm] plot [smooth] coordinates {(3,
                        1) (3.7, -0.2) (5.5, -1)};
                        \node at (4.5, 1) {{\Huge $\ioneX(\varphi)$}};
                    \end{scope}

                    \path[->,line width=1pt] (-1,0) edge (-3,0);
                    \node (v3) at (-2, 1.2) {{\Huge $\isection{\varX}$}};
                    \path[->,line width=1pt] (-3,0.5)[bend left] edge (-1,0.5);
                    \node (v2) at (-2, -0.5) {{\Huge $\iinfty{\varX}$}};
                    \begin{scope}[shift={(-12, 0)}]
                        \draw plot [smooth cycle] coordinates {(0, 0) (2, 1.4) (6,
                            2) (8,
                        -1) (3.5, -2)};
                        \draw[blue, line width=0.7mm] plot [smooth] coordinates
                        {(4.5, 1.4) (5.2, 0.2) (5.9, 0.2) (6.9, -0.3)};
                        \node[blue] at (6.9, -0.8) {{\Huge $\iinftyX(\varphi)$}};
                    \end{scope}
            \end{tikzpicture}}\]
            \caption{The natural maps associated to $\Xplus$. The blue curves
            in the middle denote the $\Gbar$-orbits, while the transversal
        black curve is $S$ and the thick blue curve is the limit of $S$.}\label{eq:diagramWhatIsGoingOn}
        \end{figure}

    To obtain the classical positive (resp.~negative) \BBname{}
    decomposition~\cite{BialynickiBirula__decomposition} we take a smooth
    proper $X$
    and pairs $(\Group, \Gbar) = (\Gmult, \mathbb{A}^1 = \Gmult\cup \{0\})$ and $(\Group,
    \Gbar) = (\Gmult, \mathbb{A}^1 = \Gmult\cup \{\infty\})$ respectively.
    For a connected linearly reductive group $\Group$, the functor $\Xplus$ is
    represented by a scheme, as proven
    in~\cite{jelisiejew_sienkiewicz__BB} and $\Xplus$ is smooth for smooth
    $X$. The proof of representability proceeds in three steps
    \begin{enumerate}[label=(\Roman*)]
        \item\label{stepMain:one} introduce a formal version $\Xhat$ of the functor $\Xplus$ and prove its
            representability. The stage for this part is the formal
            neighbourhood of $X^{\Group}$, hence the question becomes
            essentially affine and the representation theory of $\Group$ plays a
            central role,
        \item prove that the formalization map $\Xplus\to \Xhat$ is an
            isomorphism for affine schemes,
        \item\label{stepMain:three} for a $\Group$-scheme $X$ find an affine $\Group$-equivariant
            \'etale cover of fixed points of $X$ and use an easy descent
            argument to show that the natural map $\Xplus\to \Xhat$ is an
            isomorphism. The existence of such a cover is proven in~\cite{AHR}, which
            crucially depends on the linear reductivity of $\Group$, see~\cite[Prop~3.1]{AHR}.
    \end{enumerate}
    In positive characteristic the only connected linearly reductive groups
    are tori. Thus it is a natural question whether one could extend those
    existence results to reductive groups in positive characteristic.
    This seems also interesting from the point of view of geometric
    representation theory, similarly to how the $\Gmult$-case is used in~\cite{Drinfeld_Gaitsgory}.

    The aim of the present article is to prove that $\Xplus$ is representable
    for a large class of algebraic monoids $\Gbar$ with zero, so called Kempf
    monoids.
    An algebraic monoid $\Gbar$ with zero is a \emph{Kempf monoid} if there exists a central one-parameter
    subgroup $\Gmultkkbar \to Z(\Group)$ such that the induced
    map $\Gmultkkbar \to \Gbar$ extends to a map
    $\mathbb{A}^1_{\kkbar}\to \Gbar$ that sends $0$ to $0_{\Gbar}$.
    We prove that every monoid $\Gbar$ with
    zero and with reductive unit group
    is a Kempf monoid. We stress that we make no assumptions on the characteristic.
    \emph{Reductive} means as usual that the unipotent radical of $\Group$ is
    trivial. There are plenty of Kempf monoids with non-reductive unit
    group as well, such as the monoid of upper-triangular matrices.
    The main result of this paper is the following representability
    result.
    \begin{theorem}\label{ref:intro:mainthm}
        Let $\Gbar$ be a Kempf monoid with zero and with unit group $\Group$ and $X$
        be a Noetherian $\Group$-scheme over $\kk$. Then the functor $\Xplus$ is representable and
        affine of finite type over $X^{\Group}$.
    \end{theorem}
    This directly extends the previous results for the one-dimensional torus~\cite{Drinfeld, AHR}
    and linearly reductive groups~\cite{jelisiejew_sienkiewicz__BB, AHR}. This
    extension is particularly far-reaching in positive characteristic, where tori
    are the only geometrically connected linear groups. It also clarifies the
    situation in general in that the complicated representation theory for $\Group$
    poses no obstructions to representability.

    In its ideas, the proof proceeds along the
    steps~\ref{stepMain:one}-\ref{stepMain:three} above. However, there are
    two fundamental problems along the way:
    \begin{itemize}
        \item the representation theory for
            $\Group$ is complicated and thus step~\ref{stepMain:one}
            requires much more care. We introduce finitely generated Serre
            subcategories of representations and heavily employ the Kempf torus
            $(\Gmult)_{\kkbar}$ inside $\Group_{\kkbar}$ and the corresponding
            $\mathbb{A}^1_{\kkbar}$ inside $\Gbar_{\kkbar}$.
        \item the analogue of~\ref{stepMain:three} is not known and it is
            clear that the current ideas are insufficient (the
            problems arising are similar to the ones with smoothness discussed
            below). We overcome this by employing Tannakian formalism of
            Hall-Rydh~\cite{HR} to get a map $i_{\Xhat}\colon
            \Xhat\to X$ mimicking the unit map $\ioneX$. In this way, $\Xhat$
            becomes a $\Gbar$-scheme with a map to $X$. This induces a
            section of the formalization map $\Xplus\to \Xhat$ and
            implies that it is an isomorphism.
    \end{itemize}
	The language of stacks is most appropriate
    for~\ref{stepMain:three}. We delegate this part to the appendix in order to make the paper more
    accessible. The Tannaka duality in the required generality follows from
    the results of~\cite{HR}. However~\cite{HR} needs to be applied with care:
    the main results of that article require additional assumptions, so to deduce
    our claim we go into details of their argument.

    In the smooth case, we can say a little more about the morphism $\iinftyX$.
    \begin{proposition}\label{ref:intro:mainSmooth}
        Let $x\in X^{\Group}$ be such that $\iinftyX\colon \Xplus\to
        X^{\Group}$ is smooth at $\isectionX(x)\in \Xplus$. Then locally near
        the point $x$, the map $\iinftyX$ is an affine space fiber bundle with
        an action
        of $\Gbar$ fiberwise.
    \end{proposition}
    It would be desirable to have Proposition~\ref{ref:intro:mainSmooth} for
    every smooth $X$, without any assumptions on $\Xplus$. However, this is
    still open. The main problem is that the
    linear map $\mm_x\to \mm_x/\mm_x^2$ may not have a
    $\Group$-equivariant splitting, so the regularity of $x\in X$ does not
    immediately imply the regularity of
    $x\in \Xplus$, see Example~\ref{ex:problemWithSmoothness}. This is the same issue which implies that $X^{\Group}$ may
    not be smooth for smooth $\Group$-schemes $X$. Indeed, for example $\SL_p$
    acting on itself by conjugation has fixed points $\mu_p$, which is
    non-reduced. In general~\cite{Fogarty_Norman} shows that smoothness of
    fixed
    points for actions on smooth varieties characterizes linearly reductive groups. Very curiously, it seems
    that very few examples of non-smooth fixed points are known and all known
    examples seem to have smooth underlying reduced schemes.
    Also, by Lemma~\ref{ref:degreeZeroImpliesTrivialAction:lemma} below and
    affineness of $\iinftyX$ we have $X^{\Group}_{\kkbar} =
    (\Xplus)^{\Group}_{\kkbar} =
    (\Xplus)^{\Gmult}_{\kkbar}$, so we cannot hope for $\Xplus_{\kkbar}$ to be
    smooth in general, as $\Gmult$-fixed points would be smooth while
    $X^{\Group}_{\kkbar}$ can be singular. However, we do not have examples of
    smooth $X$ such that $\iinftyX\colon \Xplus\to X^{\Group}$ is not smooth.

    \section*{Acknowledgements}

        We thank Torsten Wedhorn and Andrew Salmon for enquiring about
        extending the \BBname{} decomposition to positive characteristic
        that encouraged us to write this paper. We also thank Michel Brion
        and Andrew Salmon for helpful remarks on the preliminary version of
        this paper. We thank the anonymous referee for careful reading and
        pointing out several imprecise points.

    \section{Monoids and affine schemes}

    Throughout, we fix a base field $\kk$. We do not impose any
    characteristic, algebraically closed, perfect or other assumptions on $\kk$.
    An \emph{algebraic monoid} is a geometrically integral affine variety $\Gbar$ together with an
    associative multiplication $\mu\colon\Gbar \times \Gbar \to \Gbar$ that
    has an identity element $1\in \Gbar(\kk)$. The \emph{unit group} of
    $\Gbar$ is the open subset consisting of all invertible elements. We
    denote it by $\Gbar^{\times}$ or simply by
    $\Group$. The unit group $\Group$ is open in $\Gbar$, so it is dense and connected.
    The monoid is \emph{reductive} if $\Group$
    is reductive, i.e., $\Group_{\kkbar}$ contains no
    nonzero normal unipotent subgroups.
    Every $\Group$-\emph{representation} is assumed to be
    rational,
    that is, a union of finite dimensional subrepresentations on which $\Group$ acts
    regularly.
    A canonical reference for reductive
    groups in all characteristics is~\cite{SGA3}, for a summary
    see~\cite{Demazure__thesis}. Great introductions to algebraic monoids are
    for example~\cite{Brion__On_algebraic_semigroups_and_monoids, Renner__Linear_algebraic_monoids}.
    We remark that the affineness assumption on $\Gbar$ is redundant for reductive $\Group$:
    it follows by descent from \cite[Lemma~2]{Rittatore__algebraic_monoids} that every
    geometrically integral variety over $\kk$ with an associative
    multiplication and with dense, reductive unit group is affine.

    \subsection{Kempf monoids}\label{ssec:KempfMonoids}

        Let $\Gbar$ be a Kempf monoid with unit group $\Group$ and
        $\Gmultkkbar\to \Group$ be the Kempf torus.
        The \emph{Kempf line} is the corresponding map
        $\mathbb{A}^1_{\kkbar}\to \Gbar$.
        Let $T$ be the image of $\Gmultkkbar$ in $\Group$. This
        is a one-dimensional algebraic group and it is geometrically
        connected, since it is connected and has a rational point $1_{T}$,
        see~\cite[Tag~04KV]{stacks_project}. Therefore, $T$ is a torus; in
        particular it is linearly reductive.

        \begin{lemma}\label{ref:degreeZeroImpliesTrivialAction:lemma}
            Let $X = \Spec(A)$ be a $\Gbar_{\kkbar}$-scheme such that
            the action of the Kempf torus on $X$ is trivial.
            Then the
            $\Gbar_{\kkbar}$-action on $X$ is trivial.
        \end{lemma}
        \begin{proof}
            Since $1_{\Group}, 0_{\Gbar}$ are in the closure of Kempf's torus,
            the assumption implies in particular that the
            actions of $1_{\Group}$ and $0_{\Gbar}$ are equal. For every
            $\kkbar$-scheme $S$
            and $S$-points
            $x\in X(S)$ and $g\in \Gbar(S)$ we have $gx = (g\cdot
            1_{\Group})x = g\cdot (1_{\Group}\cdot x) = g\cdot (0_{\Gbar}\cdot
            x) = (g\cdot 0_{\Gbar})\cdot x =
            0_{\Gbar}\cdot x = 1_{\Group}\cdot x$, whence the action is trivial.
        \end{proof}
        \begin{corollary}\label{ref:gradingPositive:cor}
            Let $X = \Spec(A)$ be a $\Gbar$-scheme. Consider the $\mathbb{N}$-grading on
            $A_{\kkbar}$ associated to the Kempf torus. Then
            $(A_{\kkbar})_{>0}$ is the ideal of $X^{\Group}_{\kkbar}$ in
            $X_{\kkbar}$.
        \end{corollary}
        \begin{proof}
            Since $\Gbar$ acts trivially on $X_{\kkbar}^{\Group}$, also the Kempf line
            acts trivially, so $I(X^{\Group}_{\kkbar}) \supset (A_{\kkbar})_{>0}$.
            Conversely, the ideal $(A_{\kkbar})_{>0}$ is
            $\Group_{\kkbar}$-stable since the Kempf torus is central, so it
            is $\Gbar_{\kkbar}$-stable and so $Z = V((A_{\kkbar})_{>0})
            \subset X_{\kkbar}$ is a $\Gbar_{\kkbar}$-scheme with a trivial
            action of the Kempf torus. By
            Lemma~\ref{ref:degreeZeroImpliesTrivialAction:lemma} also the
            $\Gbar$-action on $Z$ is trivial, so $Z \subset X^{\Group}$ and so
            equality holds.
        \end{proof}

    \subsection{Reductive monoids are Kempf}

    To prove the result in the title, we need a notion slightly weaker than an algebraic
    monoid. In this subsection a \emph{monoid} $M$ is an affine scheme of finite type with
    an identity and an associative multiplication (this notion is not used
    elsewhere in the article).
    \begin{lemma}\label{ref:SemisimpleMonoidsHaveClosedUnitGroup:lem}
        Let $\Group$ be a semisimple group and $\Gbar$ be a monoid with
        $\Group$ contained in the unit
        group of $\Gbar$. Then $\Group$ is closed in $\Gbar$.
    \end{lemma}
    \begin{proof}
        Since $\Gbar$ is affine of finite type, it has a faithful finite-dimensional
        representation. We fix such a representation $V$ and the associated
        embedding $i\colon\Gbar \subset \End_{\kk}(V)$. The group $i(\Group)$
        is an image of semisimple group, hence it has no non-zero characters
        so in particular $\det_{|i(\Group)}$ is trivial and so we have $i(\Group)
        \subset \SL(V)$. The map $i\colon\Group \to \SL(V)$ is a group homomorphism,
        so its image is closed. Also $\SL(V) = (\det = 1)$ is closed in
        $\End(V)$. Summing up, the group $i(\Group)$ is closed in $\End(V)$ so also in
        $\Gbar$.
    \end{proof}
    \begin{corollary}\label{ref:trivialMonoidsForSemisimpleGroups:cor}
        In the setting of
        Lemma~\ref{ref:SemisimpleMonoidsHaveClosedUnitGroup:lem} assume moreover
        that $\Gbar$ has a zero and that $\Group \subset \Gbar$ is dense. Then
        $\Gbar$ is the zero monoid.
    \end{corollary}
    \begin{proof}
        Lemma~\ref{ref:SemisimpleMonoidsHaveClosedUnitGroup:lem} shows that
        $\Group$ is closed in $\Gbar$. As it is dense, we have $\Gbar =
        \Group$. But then $0_{\Gbar}\in \Group$ is an invertible element and
        this happens only if $\Group = \{0_{\Gbar}\}$.
    \end{proof}

    \begin{proposition}\label{ref:centertrick:prop}
        Let $\Gbar$ be a reductive monoid with zero. Let $\Group$ be its unit
        group and let $Z$ be the connected component of the identity in the center
        $Z(\Group)$ with its reduced structure.
        Then for every point $g\in \Gbar(\kk)$, the point $0_{\Gbar}$ lies in
        the closure of $Z\cdot g$.
    \end{proposition}
    \begin{proof}
        Since $\Group$ is reductive, $Z$ is an algebraic
        torus~\cite[XII.4.11]{SGA3}. Consider the quotient variety
        \[
            Q := \Gbar\goodquotient Z = \Spec\left(H^0(\Gbar,
            \OO_{\Gbar})^{Z}\right)
        \]
        and the quotient map $\pi\colon \Gbar\to Q$.
        Since $Z$ is a torus, this is a categorical quotient, so the map $\pi\circ \mu\colon\Gbar \times
        \Gbar\to \Gbar\to Q$ descends to a map $\bar{\mu}\colon Q\times Q\to
        Q$ which makes $Q$ a monoid with zero.
        The map $\pi$ and the inclusion $\Group \to \Gbar$ are both dominant,
        hence so is their composition, which implies that the subgroup $\pi(\Group) \subset
        Q$ is dense. The group $\Group/Z$ is
        semisimple~\cite[XXII,4.3.5]{SGA3}, so also its image $\pi(\Group)$ is
        semisimple. Now the monoid $Q$ satisfies assumptions of
        Corollary~\ref{ref:trivialMonoidsForSemisimpleGroups:cor}, so $Q$ is a
        point.

        The points of $Q =
        \Gbar\goodquotient Z$ correspond to closed $Z$-orbits in $\Gbar$, so there
        is only one closed orbit and it is equal to
        $\{0_{\Gbar}\}$. By general theory, the closure of every $Z$-orbit
        in $\Gbar$ contains this closed orbit. This concludes the proof.
    \end{proof}

    \begin{corollary}[Reductive monoids are Kempf]
        Let $\Gbar$ be a reductive monoid with zero.
        Then $\Gbar$ is a Kempf monoid.
    \end{corollary}
    \begin{proof}
        \newcommand{\Zbar}{\overline{Z}}%
        Let $Z := Z(\Group)$ be the connected component of identity in the center of the unit group $\Group$ of
        $\Gbar$. Since $\Group$ is reductive, $Z$ is an algebraic
        torus. Let $\Zbar$ be the closure of $Z$ in $\Gbar$.
        By Proposition~\ref{ref:centertrick:prop} it contains the point
        $0_{\Gbar}$, which is a fixed point of the torus $Z$. Then
        the cone corresponding, in toric varieties sense, to the normalization of $\Zbar_{\kkbar}$ is pointed, so
        any general one-parameter subgroup $\Gmultkkbar\to Z_{\kkbar}$ extends to
        $\mathbb{A}^1_{\kkbar}\to \Zbar_{\kkbar}$ that sends $0$ to $0_{\Gbar}$.
    \end{proof}

    \begin{remark}
        Upon completion of this paper, we learned that
        over an algebraically closed field the existence of Kempf torus is proven in
        \cite[Proposition~4]{Rittatore__algebraic_monoids}. We thank Michel
        Brion for the reference.
    \end{remark}
    \subsection{Representability for affine schemes: general setup}

        In this section we make no reductivity assumptions on $\Group$.
        Fix an algebraic monoid $\Gbar$ with unit group $\Group$.
        Let $V$ be a $\Group$-representation. We say that $V$ is a
        \emph{$\Gbar$-representation} if the $\Group$-action extends to a map $\Gbar
        \times V \to V$.
        If $V$ is finite dimensional, then by the coaction map it is a $\Gbar$-representation
        if and only if there exists an embedding of $V$ into $H^0(\Gbar,
        \OO_{\Gbar})^{\oplus \dim V}$. A $\Group$-representation that is a
        subrepresentation or quotient of $\Gbar$-representation is a
        $\Gbar$-representation itself. Also the tensor product of
        $\Gbar$-representations is a $\Gbar$-representation, but the dual of a
        $\Gbar$-representation is not necessarily a $\Gbar$-representation.

        \begin{lemma}\label{ref:GbarReprs:lem}
            Let $V$ be a $\Group$-representation. Then there exists a
            quotient $V\onto W$ of $\Group$-representations such that
            \begin{enumerate}
                \item $W$ is a $\Gbar$-representation,
                \item for every other $\Group$-equivariant map $V\to W'$ with
                    $W'$ a $\Gbar$-representation there exists a
                    unique factorization $V\to W\to W'$ where $W\to W'$ is a
                    $\Gbar$-equivariant map.
            \end{enumerate}
            In other words, the quotient $V\to W$ represents the functor
            $\RepGbar\to \Set$ given by $\Hom_{\Group}(V, -)$.
        \end{lemma}
        \begin{proof}
            Consider first the case when $V$ is finite-dimensional and view it
            as $V = \Spec \Sym(V^{\vee})$. The coaction map becomes
            $\sigma_V\colon V^{\vee} \to H^0(\Group, \OO_{\Group})\tensor
            V^{\vee}$.
            Let $U_{-1} := V^{\vee}$ and let $U_0$ be the pullback defined by
            the diagram
            \[
                \begin{tikzcd}
                    U_0 \ar[r, hook]\ar[d, hook]& H^0(\Gbar, \OO_{\Gbar})\tensor_{\kk}
                    V^{\vee}\ar[d, hook]\\
                    V^{\vee} \ar[r, hook, "\sigma_V"]& H^0(\Group,
                    \OO_{\Group})\tensor_{\kk} V^{\vee}
                \end{tikzcd}
            \]
            Let $U_{n}$ for $n=1,2, \ldots $ be constructed
            inductively as the pullback
            \begin{equation}\label{eq:descentByU}
                \begin{tikzcd}
                    U_{n} \ar[r, hook]\ar[d, hook]& H^0(\Gbar, \OO_{\Gbar})\tensor_{\kk}
                    U_{n-1}\ar[d, hook]\\
                    U_{n-1} \ar[r, hook]& H^0(\Gbar, \OO_{\Gbar})\tensor_{\kk}
                    U_{n-2}
                \end{tikzcd}
            \end{equation}
            Finally let $U := \bigcap_{n} U_n$. Diagram~\eqref{eq:descentByU}
            implies that the coaction map on $U$ factors as
            \[
                U\to H^0(\Gbar, \OO_{\Gbar})\tensor_{\kk} U.
            \]
            Therefore the $\Group$-representation $W :=
            \Spec\Sym(U)$ is in fact a $\Gbar$-representation. The inclusion
            $U\into V^{\vee}$ induces a surjective map $V\to W$.
            Consider any other map $V\to W'$ to a
            $\Gbar$-representation; then replacing $W'$ by the image of $V$
            we may assume $W'$ is finite-dimensional, so $W' = \Spec \Sym(U')$
            for some $i\colon U'\to V^{\vee}$ and we obtain a commutative diagram of
            coactions
            \[
                \begin{tikzcd}
                    U'\arrow[d, "i"] \arrow[r, "\sigma_{W'}"]& H^0(\Gbar,
                    \OO_{\Gbar})\tensor_{\kk} U'\arrow[d, "\id \tensor i"]\\
                    V^{\vee} \arrow[r, "\sigma_V"]& H^0(\Group, \OO_{\Group})\tensor
            V^{\vee}
                \end{tikzcd}
            \]
            It follows from the construction of $\{U_n\}$ that $i$ factors as
            $U'\to U\to V^{\vee}$.
            Consider now the case of a general (but, as always, rational)
            $\Group$-representation $V$. For
            every finite-dimensional subrepresentation $V'$ consider the
            corresponding quotient $V'\to W(V')$ and its kernel $K(V')$.
            For an inclusion of finite-dimensional subrepresentations $V'\to
            V''$, the universal property of $W(-)$ gives a map $W(V')\to W(V'')$
            whence an inclusion $K(V') \to K(V'')$.
            The union $K = \bigcup_{V'} K(V')$ is a $\Group$-subrepresentation
            of $V$ and $V/K$ has the required universal property.
        \end{proof}
        \begin{proposition}\label{ref:representabilityForAffine:prop}
            Let $X = \Spec(A)$ be an affine $\Group$-scheme. Then $\Xplus$ is
            represented by a closed subscheme.
        \end{proposition}
        \begin{proof}
            Let $A \to A/\mathcal{G}$ be the universal quotient
            $\Gbar$-representation as in Lemma~\ref{ref:GbarReprs:lem}. Let
            $I\subset A$ be the ideal generated by $\mathcal{G}$.
            The closed scheme $\Spec(A/I)$ is $\Group$-stable and in fact its
            $\Group$-action extends to $\Gbar$.
            We claim
            that $\Xplus = V(I) \simeq \Spec(A/I)$.
            Pick a $\Gbar$-scheme $Y$ with a
            $\Group$-equivariant map $\varphi\colon Y\to X$. The map $\varphi$
            factors through the $\Gbar$-scheme $\Spec H^0(Y, \OO_Y)$, so we
            may assume that $Y$ is affine. The pullback $A\to H^0(Y, \OO_Y)$ maps
            $A$ to a $\Gbar$-representation, hence maps $\mathcal{G}$ to zero,
            and thus factors through $\Xplus$.
        \end{proof}

        Now we slightly generalize
        Proposition~\ref{ref:representabilityForAffine:prop} taking into
        account open immersions.  We say that a $\Group$-scheme $X$ is
        \emph{locally $\Group$-linear} if it is covered by $\Group$-stable
        open affine subschemes. We say that $X$ is locally $\Gbar$-linear if
        it is covered by $\Gbar$-stable open affine subschemes.
        \begin{lemma}[open immersions,~{\cite[Prop~5.2]{jelisiejew_sienkiewicz__BB}}]\label{ref:openImmersion:lem}
            Let $U\into X$ be an open immersion of $\Group$-schemes. Then the
            diagram
            \[
                \begin{tikzcd}
                    U^{+} \arrow[r, "\iinfty{U}"]\arrow[d, hook] & U^{\Group}\arrow[d, hook]\\
                    \Xplus \arrow[r, "\iinftyX"] & X^{\Group}
                \end{tikzcd}
            \]
            is cartesian.
        \end{lemma}
        \begin{proof}[Sketch of proof]
            Arguing as in the proof of Proposition~\ref{ref:centertrick:prop}
            using the Kempf torus, we see that the point $0\in \Gbar$ has no
            $\Group$-stable open neighbourhoods except the whole $\Gbar$, so
            for every scheme $S$, the locus $0 \times S \subset
            \Gbar \times S$ has no $\Group$-stable open neighbourhoods except
            the whole $\Gbar \times S$. Hence $\varphi\colon\Gbar \times S\to X$ factors
            through $U$ if and only if $\varphi_{|0 \times S}\colon S\to
            X^{\Group}$ factors through $U^{\Group}$.
        \end{proof}

        \begin{proposition}\label{ref:representabilityForLocLin:prop}
            Let $X$ be a locally $\Group$-linear scheme. Then $\iinftyX\colon
            \Xplus\to X^{\Group}$ is affine, so $\Xplus$ is represented by a
            locally linear $\Gbar$-scheme.
        \end{proposition}
        \begin{proof}
            Let $\{U_i\}$ be a $\Group$-stable open affine cover of $X$.
            By Lemma~\ref{ref:openImmersion:lem} the pullback of $\iinftyX$
            via $U_i^{\Group} \to X^{\Group}$ is the map
            $\iinfty{U}\colon U^+\to U^{\Group}$. This map is affine
            by Proposition~\ref{ref:representabilityForAffine:prop}. It
            follows that the map
            $\iinftyX$ becomes affine after the pullback to the cover
            $\bigsqcup_i U_i\to X$, so $\iinftyX$ is affine.
        \end{proof}
        \begin{remark}
            Proposition~\ref{ref:representabilityForLocLin:prop} can be stated
            with weaker assumptions: we do not really
            need $\{U_i\}$ to cover $X$, just the inclusion $\bigcup_i U_i \supset
            X^{\Group}$.
        \end{remark}

        \subsection{Representability for affine schemes: reductive case}

        Proposition~\ref{ref:representabilityForAffine:prop}, while
        satisfactory for our general purposes, says little about the resulting
        scheme. In this section we prove that one can say more in the
        reductive case.

        Fix a reductive monoid $\Gbar$ with unit group $\Group$. We
        additionally assume that the variety $\Gbar$ is normal, however we do
        not need to assume that $\Gbar$ has a zero.
        All maximal tori of $\Group$ are conjugate. We fix one such torus $T$
        and its closure $\Tbar \subset \Gbar$. We say that a
        $T$-representation is a \emph{$\Tbar$-representation} if the action of
        $T$ extends to an action of $\Tbar$. A simple $T$-representation is an
        \emph{outsider representation} if it is not a $\Tbar$-representation.

        \begin{lemma}[Extension principle]\label{ref:extensionPrinciple:lem}
            Let $V$ be a $\Group$-representation. Then the following are equivalent
            \begin{enumerate}
                \item\label{it:epone} the representation $V$ is a $\Gbar$-representation,
                \item\label{it:eptwo} the representation $V$ is a $\Tbar$-representation.
            \end{enumerate}
        \end{lemma}
        \begin{proof}
            The implication~\ref{it:epone}$\implies$\ref{it:eptwo} is trivial,
            and the converse
            is~\cite[Theorem~5.2]{Renner__Linear_algebraic_monoids}.
            (the referenced theorem requires $\Gbar$ to be normal and this is
            the main point where we use normality of $\Gbar$).
        \end{proof}

        We now prove the representability of $\Xplus$ for $X$ affine.
        \begin{proposition}\label{ref:representabilityForAffineReductiveCase:prop}
            Let $X = \Spec(A)$ be an affine $\Group$-scheme. Then $\Xplus$ is
            represented by a closed subscheme whose ideal is the smallest
            $\Group$-ideal containing all outsider representations of $T$ in
            $A$.
        \end{proposition}
        \begin{proof}
            Let us discuss the easiest case: $X$ a $\Gbar$-scheme. Let
            $\mubar\colon \Gbar \times X\to X$ denote the action. Then every
            family $\varphi_1\colon S\to X$ extends uniquely to a
            $\Group$-equivariant family $\varphi = \mubar \circ (\id \times
            \varphi_1)\colon \Gbar \times S\to X$, hence $\ioneX\colon \Xplus \to
            X$ is an isomorphism.

            \newcommand{\BBforT}{\tilde{X}}%
            Let us return to the general case.
            Consider $X$ as a $T$-variety and let $\BBforT$ be the
            \BBname{} decomposition for $T$. As $T$ is a torus, it is linearly
            reductive, so
            by~\cite[Proposition~4.5]{jelisiejew_sienkiewicz__BB} the scheme
            $\BBforT$ is a closed subscheme of $X$.
            Let $J := I(\BBforT)\subset A$ be its ideal and let $I \subset A$ be the ideal
            generated by $\Group \cdot J$.
            By~\cite[Proposition~4.5]{jelisiejew_sienkiewicz__BB} the ideal
            $J$ is generated by all irreducible
            $T$-subrepresentations of $A$ that are not $\Tbar$-representations.
            From linear reductivity of $T$ we deduce that $A/I$ is a $\Tbar$-representation.

            Let $X' = \Spec(A/I)$. By
            construction, $X'$ is the largest $\Group$-scheme contained in
            $\BBforT$.
            The coordinate ring $A/I$ of $X'$ is a $\Group$-representation
            that is also a $\Tbar$-representation, hence by Extension
            Principle~\ref{ref:extensionPrinciple:lem} it is a
            $\Gbar$-representation, so $X'$ is a $\Gbar$-scheme.
            Every $\Group$-equivariant family $\varphi\colon \Gbar \times
            \Spec(C)\to
            X$ induces a $\Group$-equivariant pullback map
            \[
                \varphi^{\#}\colon A\to H^0(\Gbar,
                \OO_{\Gbar})\tensor_{\kk} C.
            \]
            The right hand side is a
            $\Gbar$-representation, so a $\Tbar$-representation, so
            $\ker(\varphi^{\#})$ contains $J$. The pullback is
            $\Group$-equivariant, so $\ker(\varphi^{\#})$ contains $I$ as
            well. This shows that $\varphi$ factors through
            $X'\into X$ and consequently that $\Xplus = (X')^{+}$. But $X'$ is an affine
            $\Gbar$-scheme, hence the ``easiest case'' from the beginning of
            the present proof applies and gives $X'^+ = X'$.
        \end{proof}

        \begin{example}[Issues with regularity]\label{ex:problemWithSmoothness}
            Let $x\in X = \Spec(A)$ be a $\Group$-fixed $\kk$-point which we
            identify with $\isectionX(x)\in \Xplus$. Applying
            Proposition~\ref{ref:representabilityForAffineReductiveCase:prop}
            we get that for every $n\in \mathbb{N}$ the truncated local ring
            $A^{n, +} := \OO_{\Xplus, x}/\mm_{x}^{n+1}$ is the
            quotient of $A^n := \OO_{X, x}/\mm_x^{n+1}$ by the smallest $\Group$-ideal
            containing all outsider representations of $A^n$ with respect to
            $T$.

			If the group $\Group$ was linearly reductive, we could take
            a $\Group$-equivariant section $s$ of $\mm_x\onto \mm_x/\mm_x^2$ and
            conclude that $A^{n, +}$ is the quotient of $A^n$ by the
            image under $s$ of the $\Group$-representation generated by the
            $T$-outsider representations of $\mm_x/\mm_x^2$.  Consequently, the complete
            local ring $\hat{\OO}_{\Xplus, x}$ would be the quotient of
            $\hat{\OO}_{X, x}$ by a sequence of elements whose images in the
            cotangent space are linearly independent. When
            $x\in X$ is regular, we would conclude that $x\in \Xplus$ is
			regular. This would be essentially the classical
			Iversen's~\cite{Iversen__fixed_points} argument on the smoothness of fixed
			points of $\Group$ for $\Group$ acting on smooth $X$.
			However, for groups $\Group$ that are not linearly reductive Iversen's
            argument fails and so we cannot hope for the existence of section
            $s$ in our case.
        \end{example}

        \section{Formal $\Group$-schemes}\label{sec:formalSchemes}

    In this section we prove the main results for the formal \BBname{} functor
    $\Xhat$
    that we will introduce in the next section. The main advantage of the
    formal functor over $\Xplus$ is that it is defined on the affine level;
    correspondingly in this section we speak the language of algebra rather than
    geometry. Throughout, we assume that $\Gbar$ is a Kempf monoid and that $0\in
    \mathbb{N}$.

    \subsection{Setup}
    We begin with basic definitions. A \emph{formal abelian group} is a
    sequence $(A_n)_{n\in \mathbb{N}}$ of abelian groups together with
    surjections $\pi_n\colon A_{n+1}\onto A_n$.
    A morphism of formal abelian groups $(A_n)\to (A'_n)$ is a family of
    morphisms $f_n\colon A_n\to A'_{n}$ such that the diagram
    \[
        \begin{tikzcd}
            A_{n+1} \arrow[r, two heads, "\pi_n"]\arrow[d, "f_{n+1}"] & A_n\arrow[d, "f_n"]\\
            A'_{n+1} \arrow[r, two heads, "\pi_n'"] & A'_n
        \end{tikzcd}
    \]
    commutes for every $n$.
    We obtain an abelian category of formal abelian groups. Using
    abstract-nonsense for this category, we can define formal algebras (as
    algebra objects), $\Group$-actions etc. Below we gather some specific
    cases.
    For a $\kk$-algebra $B$, we say that a formal
    abelian group is a \emph{formal $B$-algebra} if all $A_n$'s are $B$-algebras
    and $\pi_n$ are maps of $B$-algebras. A \emph{module} over a formal
    $B$-algebra $(A_n)$ is a formal abelian group $(M_n)$ where additionally each
    $M_n$ is an $A_n$-module and $\pi_n$ is a homomorphism of $A_{n+1}$-modules.
    We say that a formal $B$-algebra $(A_n)$ is a $\Group$-algebra if all
    algebras $A_n$ are $\Group$-algebras, the maps $\pi_n$ are $\Group$-equivariant homomorphisms
    of algebras and moreover the structure maps $B\to A_n$ are
    $\Group$-equivariant for the trivial $\Group$-action on $B$. Similarly we
    define $\Gbar$-algebras and $\Group$- or $\Gbar$-modules. In particular,
    whenever we speak about a $\Group$-action on a module $(M_n)$ over a
    $\Group$-algebra $(A_n)$, we assume that the multiplication maps
    $A_n\tensor M_n\to M_n$ are $\Group$-equivariant.

    \begin{definition}[Formalization]\label{ref:formalization:def}
        Let $A$ be an abelian group and $(I_n \subset A)_{n\in \mathbb{N}}$ an
        increasing sequence of subgroups. The \emph{associated formal abelian
        group} is
        $(A/I_n)_{n}$ with the natural surjections. When $A$ is a ring and
        $I_n := I^{n+1}$ for an ideal $I \subset A$ then the associated
        formal abelian group is a ring which we call the \emph{formalization
        of $(A, I)$}. When $M$ is an $A$-module, then
        $(M/I^{n+1}M)_n$ is an $(A/I^{n+1})_n$-module called the
        \emph{formalization of $(M, I)$}.
    \end{definition}
    \begin{definition}[adic algebras and
        algebraizations]\label{ref:adicAndAlgebraization:def}
        We say that a formal algebra $\AA = (A_n)$ is \emph{adic} if for every $m \geq
        n$ the map $A_m\to A_n$ is surjective and
        \[
            \ker(A_m\to A_n) = \ker(A_m \to A_0)^{n+1}.
        \]
        This implies in particular that $\ker(A_m \to A_0)^{m+1} = 0$ for every
        $m$.
        For a formal adic algebra $(A_n)$ its \emph{algebraization} is
        an algebra $A$ with an ideal $I$ such that the formalization of $(A,
        I)$ is isomorphic to $(A_n)$. A module $(M_n)_n$ over an adic
        algebra $(A_n)_n$ is \emph{adic} if the maps $M_{n}\to M_{n-1}$ are surjective and for every
        $m\geq n$ we have $\ker(M_m\to M_n) = \ker(A_m \to A_n) M_m$ so that the
        maps induce isomorphisms $M_m\tensor_{A_m} A_n\to M_n$.
    \end{definition}
    Formalizations of algebras are adic. In the next sections we will prove
    that in certain situations ``standard'' adic algebras admit
    algebraizations.

        \subsection{Serre subcategories of $\Group$-linearized sheaves}
    We will be interested in constructing algebraizations of a given formal
    algebra in an $\Group$-equivariant way (see Section~\ref{ssec:GequivAlg}). If $\Group$
    was linearly reductive, this would mean that we just look separately at each
    $\lambda$-isotypic component for a simple $\Group$-representation $\lambda$.
    But the category of $\Group$-representations is far from semisimple, so
    working with irreducible representations makes little sense.
    We need a
    generalization of them. While in the linearly reductive case
    every finite dimensional representation is a direct sum of simple
    representations, here every representation has a filtration with simple
    subquotients.
        \begin{lemma}[Jordan-H\"older]\label{ref:JordanHolder:lem}
            Let $V$ be a finite dimensional $\Group$-representation. Then there exists
            a filtration $0 = V_0 \subset V_1 \subset \ldots \subset V_{r+1} = V$
            with simple subquotients $E_i = V_{i+1}/V_{i}$. The set
            $\{E_0, \ldots ,E_r\}$ does not depend on the filtration chosen.
            The representations $E_0, \ldots ,E_r$ are called the \emph{composition
            factors} of $V$.
        \end{lemma}
        \begin{proof}
            This follows from the Jordan-H\"older theorem.
        \end{proof}
        In the linearly reductive case, for a simple representation $\lambda$
        and a finite dimensional representation $V$, one has the isotypic
        component $V[\lambda] \subset V$ associated
        to $\lambda$: the largest
        subrepresentation that is a direct sum of $\lambda$'s.
        In our situation, for a set of simple $\Group$-representations $\lambda = \{\lambda_1, \ldots
        ,\lambda_r\}$ and a representation $V$ we define $V[\lambda]
        \subset V$ to be the largest subrepresentation whose composition
        factors belong to $\{\lambda_1, \ldots ,\lambda_r\}$.
        To put this idea on a solid footing we use
        Serre subcategories.

        \begin{definition}
            Let $C$ be an abelian category. A
            \emph{Serre subcategory} $S$ is a full subcategory closed under direct
            sums and such that for any short exact sequence
            \[
                0\to V_1 \to V_2\to V_3\to 0
            \]
            we have $V_2\in S$ if and only if $V_1, V_3\in S$.
        \end{definition}

        \begin{lemma}\label{ref:Serresubobject:deflem}
            Let $S \subset C$ be a Serre subcategory and $V\in C$ be an
            object. Then there exists a unique largest
            subobject of $V$ that
            lies in $S$. We denote it by $V[S]$.
            For a morphism $f\colon V_1\to V_2$ in $C$ the image $f(V_1[S])$ lies
            in $V_2[S]$, so we get an induced morphism $f[S]\colon V[S] \to
            W[S]$.
            For an exact
            sequence
            \[
                \begin{tikzcd}
                    0 \ar[r] & V_1\ar[r] & V_2\ar[r] & V_3
                \end{tikzcd}
            \]
            we get an exact sequence $0\to V_1[S] \to V_2[S] \to
            V_3[S]$, so the functor $(-)[S]$ is left exact.
        \end{lemma}
        \begin{proof}
            Consider the family $\{W_i\}$ of all the subobjects of $V$
            that lie in $S$. Then $\bigoplus W_i$ lies in
            $S$, hence its quotient $\sum W_i \subset V$ lies in
            $S$. Take $V[S] := \sum W_i$.  For $f\colon V_1\to V_2$ the object
            $f(V_1[S])$ is a quotient of $V_1[S]$, hence $f(V_1[S])\in S$, so
            $f(V_1[S]) \subset V_2[S]$ by definition of $V_2[S]$.
            Finally, $\ker(V_2[S] \to V_3[S])$ is in $S$ as a subobject of
            $V_2[S]$ and is a subrepresentation of $V_1$ since the sequence is exact. Thus
            $\ker(V_2[S]\to V_3[S]) \subset V_1[S]$ which proves exactness.
        \end{proof}
        It is in general not true, even in our setup, that
        $(-)[S]$ is right exact.

        In the following, we will use the two main examples of Serre
        subcategories.
        \begin{example}
            Fix a set $\lambda = \{\lambda_1, \ldots ,\lambda_r\}$ of simple
            $\Group$-representations and consider the full
            subcategory of $\RepG$ consisting of representations $V$ that are
            unions of finite dimensional subrepresentations with composition
            factors in $\lambda$. This is a Serre category, which we
            denote by $\RepGlambda$ and call the \emph{Serre category generated
            by $\lambda$}. We denote by $(-)[\lambda]$ the associated functor.
        \end{example}
        For a $\Group$-representation $V$ the inclusion of sets induces a
        partial order on the set of all $\lambda$'s which
        makes the set $\{V[\lambda]\}_{\lambda}$ a direct system and $V$
        its colimit.
        \begin{remark}\label{ref:strongRationality:remark}
            For a representation $V$, the subspace $V[\lambda]$ is a
            subrepresentation, in particular it is rational. But also a stronger ``global''
            rationality condition is satisfied: there exists a finite
            dimensional $\kk$-linear subspace $W \subset H^0(\Group, \OO_{\Group})$
            such that the coaction map for any $V[\lambda]$ has image in $W
            \tensor V[\lambda] \subset H^0(\Group,
            \OO_{\Group})\tensor_{\kk} V[\lambda]$. Indeed, let $W$ be the
            minimal subspace such that the coactions of all (finitely many!)
            simple representations from $\lambda$ have image in
            $W\tensor(-)$. Then $W$ satisfies the assumptions.
        \end{remark}

        \begin{example}
            Let $T \subset Z(\Group)$ be a central torus (which means that
            $T_{\kkbar}  \simeq \mathbb{G}_{m,\kkbar}^{m}$ for some $m$, in
            particular $T$ is linearly reductive).
            Let $\Chi$ be a finite set of simple $T$-representations
            and consider the full subcategory of $\RepG$ consisting of
            $\Group$-representations $V$ which as a $T$-representation
            are direct sums of elements of $\Chi$. This is a Serre category that we denote by
            $\RepG[\Chi]$. We denote by $(-)[\Chi]$ the associated functor.
        \end{example}
        \begin{remark}\label{ref:splittingForGmult:remark}
            Since $T$ is central, for every its simple representation $\chi$
            and every $\Group$-representation $V$ the isotypic component
            $V[\chi]$ is a $\Group$-subrepresentation, therefore so is $V[\Chi] =
            \bigoplus_{\chi\in \Chi} V[\chi]$. (we see here a clash of
            notation between isotypic components and Serre categories;
            fortunately both notations agree). The
            functor $(-)[\Chi]$ is exact.
        \end{remark}
            The constructions $(-)[\lambda]$ and $(-)[\Chi]$ are connected as
            follows. For $\lambda$, let $\Chi(\lambda)$ be
            the set of all $T$-weights that appear in $\bigoplus_{\lambda_i\in
            \lambda} \lambda_i$. For a $\Group$-representation $V$, we
            have $V[\lambda] \subset
            V[\Chi(\lambda)]$. On the other hand, if $w\in V[\Chi(\lambda)]$, then we may form
            $\mu(w) = \left\{ \mu_1, \ldots ,\mu_s \right\}$ where $\mu_i$ are
            composition factors of the representation generated by $w$. Since
            $\Group\cdot w \subset V[\Chi(\lambda)]$, we have
            $\Chi(\mu(w))\subset \Chi(\lambda)$. By definition $w$ lies in
            $V[\mu(w)]$,
            therefore we obtain
            \begin{equation}\label{eq:ChiFilledByMu}
                V[\Chi] = \bigcup_{\mu\colon \Chi(\mu) \subset \Chi} V[\mu].
            \end{equation}

            \subsection{$\Group$-equivariant algebraizations}\label{ssec:GequivAlg}

    In this subsection we fix a torus $T \subset
    Z(\Group)$, not necessarily split. We assume that the Kempf torus
    $(\Gmult)_{\kkbar}\to \Group$ factors through $T$. Below we use the
    convention that application of $(-)[\lambda]$ precedes taking the limit,
    so that $\lim_n M_n[\lambda] := \lim_n (M_n[\lambda])$ and similarly for
    $\Chi$.

    \begin{definition}
    Let $(M_n)$ be a formal abelian group with a $\Group$-action. Its
        \emph{$\lambda$-algebraization} is
        \[
            \lim_{\Group}(M_n) := \colim_{\lambda} \lim_{n} M_n[\lambda],
        \]
    \end{definition}
    By Remark~\ref{ref:strongRationality:remark}, $\lim_{\Group}(M_n)$ is a
    $\Group$-representation. It is
    the limit of the diagram $ \ldots \onto M_{n+1}\onto M_n\onto  \ldots
    \onto M_0$ in the category of abelian groups with $\Group$-action; hence the notation.
    If $\AA = (A_n)_n$ is a formal algebra, then $\lim_{\Group}(\AA)$ is an
    algebra as well. Indeed for $\lambda$ and $\mu$ let $\lambda\tensor \mu$
    denote the Serre subcategory generated by composition factors of all
    $\{\lambda_i \tensor \mu_j\}_{\lambda_i\in \lambda, \mu_j\in
    \mu}$. Then the multiplication on an $A_n$ restricts to
    $A_n[\lambda]\tensor A_n[\mu]\to A_{n}[\lambda \tensor \mu]$ and induces
    multiplication $(\lim_{n} A_n[\lambda]) \tensor (\lim_n A_n[\mu]) \to \lim_n
    A_n[\lambda\tensor \mu]$ and thus a multiplication on
    $\lim_{\Group}(\AA)$.
    \begin{remark}
        The reader might wonder what is the connection between algebraizations
        from Definition~\ref{ref:formalization:def} and
        $\lambda$-algebraizations. We will see in
        Theorem~\ref{ref:existenceOfAlgebraization:thm} that in favorable
        conditions a $\lambda$-algebraization is an algebraization.
    \end{remark}

    \begin{definition}
    Let $(M_n)$ be a formal abelian group with a $T$-action. Its
        \emph{$\Chi$-algebraization} is
        \[
            \lim_{T}(M_n) := \colim_{\Chi} \lim_{n} M_n[\Chi],
        \]
        where the colimit is taken over all finitely generated Serre
        subcategories of $\Rep_T$.
    \end{definition}

        \begin{lemma}\label{ref:leftExactnessOfAlgebraization:lem}
            Let $0\to \cM\to \cN\to \cP$ be an exact sequence of formal
            abelian groups with $\Group$-action (resp.~$T$-action). Then the
            induced sequence $0\to \lim_{\Group}(\cM) \to
            \lim_{\Group}(\cN)\to \lim_{\Group}(\cP)$ is exact (resp., the
            induced sequence of $\lim_{T}(-)$ is exact).
        \end{lemma}
        \begin{proof}
            This follows from left exactness of $(-)[\Chi]$ and
            $(-)[\lambda]$ and the fact that colimits in both algebraizations
            are taken over filtered sets.
        \end{proof}

    Since $\Rep_T$ is semisimple we have a canonical isomorphism
    \[
        \lim_{T}(M_n) = \bigoplus_{\chi} \lim_{n}M_n[\chi],
    \]
    where $\chi$ runs through all simple
    $T$-representations. It follows that
    \begin{equation}\label{ref:ChiOnChiAlgebraization:eq}
        (\lim_T(M_n))[\Chi] = \lim_{n} M_n[\Chi]
    \end{equation}
    for every set of simple $T$-representation $\Chi$.
    The $T$-module $\lim_{T}(M_n)$ is the limit of $ \ldots \onto M_{n+1}\onto M_n\onto  \ldots
    \onto M_0$ in the category of abelian groups with $T$-action.
    For a formal abelian group $(M_n)$ with $\Group$-action, the restriction $\RepG\into
    \Rep_{T}$ induces a injective \emph{$\lambdachi$-comparison map}:
    \begin{equation}\label{eq:lambdachi}
        \lim_{\Group}(M_n)\into \lim_{T}(M_n)
    \end{equation}
    that embeds $M_n[\lambda]$ into $M_n[\Chi(\lambda)]$ and so $\lim_n
    M_n[\lambda]$ into $\lim_n M_n[\Chi(\lambda)]$.
    The reason to introduce both algebraizations is that
    $\lambda$-algebraization is more canonical --- it comes with a
    $\Group$-action ---  but also more challenging to
    work with, mostly
    because $[-](\lambda)$ is not right-exact. We
    will prove that the $\lambdachi$-comparison map is an isomorphism in
    situations of interest, see Corollary~\ref{ref:algebraizationsAgree:cor}.
    For this we employ
    various stabilization results that we now prove.

    \subsection{Stabilization}\label{ssec:stabilization}
    We keep the setup from the previous subsection
    (Subsection~\ref{ssec:GequivAlg}).

    We say that a formal abelian group $(M_n)$ with a $\Group$-action
    \emph{stabilizes $\Group$-equivariantly} if for every $\lambda$ the map
    $M_{n+1}[\lambda]\to M_{n}[\lambda]$ is an isomorphism for all $n$ large
    enough. Similarly, we say that a formal abelian group $(M_n)$ with a
    $T$-action \emph{stabilizes $T$-equivariantly} if for every
    $\Chi$ the map $M_{n+1}[\Chi]\to M_{n}[\Chi]$ is an isomorphism for all
    $n$ large enough. If $(M_n)$ has a $\Group$-action and stabilizes $T$-equivariantly then it also
    stabilizes $\Group$-equivariantly. Indeed, for every $\lambda$ the map
    $M_{n+1}[\Chi(\lambda)]\to M_{n}[\Chi(\lambda)]$ is an isomorphism for
    $n\gg 0$ and so using~\eqref{eq:ChiFilledByMu} we deduce that also the map
    \[
        M_{n+1}[\lambda] = M_{n+1}[\Chi(\lambda)][\lambda]\to
        M_{n}[\Chi(\lambda)][\lambda] = M_n[\lambda]
    \]
    is an isomorphism.
    We now give an instance where those stabilizations hold.
    \begin{definition}[standard adic formal algebra]\label{ref:standardadic:def}
        We say that a formal $\Gbar$-algebra $\AA$ is a \emph{standard
        adic formal $\Gbar$-algebra} if all of the following hold
        \begin{enumerate}
            \item $\AA$ is an adic algebra (as in
                Definition~\ref{ref:adicAndAlgebraization:def}),
            \item $\AA$ is an $A_0$-algebra (which means that every $A_n$ is
                an $A_0$-algebra and $A_n\to A_{n-1}$ are
                surjections of $A_0$-algebras),
            \item\label{it:standardadicinvariants} $\Spec(A_n)^{\Group} \to \Spec(A_0)$ is an isomorphism for every $n$.
        \end{enumerate}
    \end{definition}
    A formal algebra obtained from a $\Gbar$-scheme $X = \Spec(A)$ with $I =
    I(X^{\Group})$ is standard adic. We will see in
    Section~\ref{ssec:algebraization_geometric} that under good
    conditions standard adic algebras come from $\Gbar$-schemes.

    Let $\MM = (M_n)_n$ be a module over a standard adic algebra $(A_n)_n$. If the module $\MM$ is
    equipped with a $T$-action then we say that $\MM$ is
    \emph{grounded} if there exists a finite set of $T$-characters $\Chi$ such
    that $M_0 = M_0[\Chi]$.
    By Definition~\ref{ref:standardadic:def}\eqref{it:standardadicinvariants}
    for $n=0$ the $\Group$-action on $A_0$ is trivial, so
    the module $\MM$ is grounded whenever $M_0$ is a finitely generated
    $A_0$-module.
    \begin{example}\label{ex:groundness}
        Let $T$ be $\Gmult$ and let $\AA = (\kk[t]/t^n)_n$ be equipped with a $T$-action coming
        from the standard grading. The
        $\AA$-modules $\MM^1 = (\kk[t]/t^n\oplus (\kk[t]/t^n)t^{-1})_n$ and
        $\MM^2 =
        (\bigoplus_{i\in \mathbb{N}}(\kk[t]/t^n)t^{-1})_n$ are grounded, while
        $\MM^3 = \left( \bigoplus_{i\in \mathbb{N}}
        (\kk[t]/t^n)t^{-i}\right)_n$ is not grounded.
        The difference is that $\MM^3$ has generators of arbitrarily negative
        weights, while the generators $\MM^2$ are infinite, but their weights
        are bounded from below.
    \end{example}

    \begin{definition}[stabilization indices]
        For a simple $T$-representation $\chi$, we define $n_{\chi}\in
        \mathbb{Z}$ as the
        maximal weight of the Kempf torus in $\chi_{\kkbar}$. We set
        $n_{\Chi} = \max(n_{\chi}\ :\ \chi\in \Chi)$ and $n_{\lambda} :=
        n_{\Chi(\lambda)}$.
    \end{definition}
    For example, if $\chi$ is a trivial $T$-representation then $n_{\chi} = 0$.

        \begin{lemma}[$T$-equivariant stabilization
            lemma]\label{ref:stabilizationAlgebraically:lem}
            Let $\AA = (A_n)$ be a standard adic formal $\Gbar$-algebra. Let $\MM =
            (M_n)$ be an adic $\AA$-module with a $T$-action. Assume that
            $\MM$ is grounded. (For example, this is the case when $M_0$ is a
            finitely generated $A_0$-module.)

            \noindent Fix a finite set of characters $\Chi$. Then there exists a natural number $n_{\Chi, \MM}$ such that
            for
            every $n > n_{\Chi, M}$ the surjection
            \[
                M_{n}[\Chi] \onto M_{n-1}[\Chi]
            \]
            is an isomorphism. If $\MM = \AA$ then we may take $n_{\Chi,\MM} =
            n_{\Chi}$.
        \end{lemma}

        \begin{proof}
            The claim is invariant under field extensions, so we pass to
            the algebraic closure of $\kk$. We also immediately reduce to the
            case $\Chi = \{\chi\}$.
            Let $m$ be the minimal weight of the Kempf torus on $M_0$; such a
            weight exists since the set of weights is finite because $\MM$ is
            grounded. We claim that $n_{\chi, \MM} = n_{\chi} + |m|$
            satisfies the assumptions.

            By Corollary~\ref{ref:gradingPositive:cor} for every $n$ the ideal
            $\ker(A_n\to A_0)$ consists entirely of positive weights. Since
            $\MM$ is adic, we have $M_0 = M_n/(\ker(A_n\to A_0)M_n)$. By the
            graded Nakayama lemma, the minimal weight of the Kempf torus on
            $M_n$ is equal to $m$.

            We have a short exact sequence
            \[
                0\to \ker(M_n\to M_{n-1})[\chi] \to M_n[\chi]\to
                M_{n-1}[\chi]\to 0
            \]
            so it is enough to prove that $\ker(M_n\to M_{n-1})[\chi] =
            0$ for all $n > n_{\chi, \MM}$.  Since $\AA$ and $\MM$ are adic,  we have
            $\ker(M_n\to M_{n-1}) = \ker(A_n\to A_{0})^nM_n$.
            If $n > n_{\chi}+|m|$ then the minimal weight of the
            Kempf torus in $\ker(A_n\to A_{0})^nM_n$ is greater than
            $n_{\chi} + |m| + m$ which in turn is greater than or equal $n_{\chi}$.
            The isotypic component $(\ker(A_n\to A_{0})^nM_n)[\chi]$ is
            thus equal to zero, as claimed.
            If $\MM = \AA$ then $m = 0$ and so $n_{\chi,\MM} = n_{\chi}$.
        \end{proof}
        As discussed at the beginning of this subsection, for a
        $\Group$-module $\MM$ its $T$-equivariant stabilization implies
        $\Group$-equivariant stabilization.
        \begin{corollary}\label{ref:algebraizationsAgree:cor}
            Let $\AA$ be a standard adic formal $\Gbar$-algebra and let $\MM$ be an
            adic $\Group$-module. Assume $\MM$ is grounded.
            Then the $\lambdachi$-comparison
            map~\eqref{eq:lambdachi} for $\MM$ is an
            isomorphism. Consequently, the map $\lim_{\Group}(M_n)\to M_n$ is
            surjective for every $n$.
        \end{corollary}
        \begin{proof}
            The comparison map is always injective, so it is enough to prove
            surjectivity.
            Let $m\in \lim_{T}(M_n)$ and choose $\Chi$ such that $m\in \lim_n
            M_n[\Chi]$. Let $n' := n_{\Chi, \MM}$. For $n> n'$ the map
            $M_{n}[\Chi]\to M_{n-1}[\Chi]$ is an isomorphism, so
            $\lim_n M_n[\Chi]\to M_{n'}[\Chi]$ is an isomorphism. Since
            $T$ is central in $\Group$, the subgroup $M_{n'}[\Chi]$
            is $\Group$-stable and so there exists a finite set of
            $\Group$-representations $\mu$ such that $m\in M_{n'}[\mu]$.
            But then $\Chi(\mu) \subset \Chi$ and by the discussion above,
            also $\lim_n M_n[\mu]\to M_{n'}[\mu]$ is an isomorphism, so
            $m$ lifts to an element of $\lim_n M_n[\mu]$
            and so $m$ is in the image of the $\lambdachi$-comparison map.
            The final claim follows because $\lim_{T}(M_n)\to M_n$ is
            surjective by right-exactness of $(-)[\Chi]$.
        \end{proof}
        \begin{theorem}[$T$-equivariant algebraization
            exists]\label{ref:existenceOfAlgebraization:thm}
            Let $\AA$ and $\MM$ be as in
            Lemma~\ref{ref:stabilizationAlgebraically:lem}.
            Let $M = \lim_{T} \MM$ and $A = \lim_{\Group} \AA$. Then $\ker(M\to M_n) =
            \ker(A\to A_0)^{n+1}M$ for all $n$. Therefore $\MM$
            admits an algebraization $(M, \ker(A\to A_0))$.
            In particular, $\AA$
            admits an algebraization $(\lim_{\Group}(\AA), \ker(A\to A_0))$.
        \end{theorem}
        \begin{proof}
            Corollary~\ref{ref:algebraizationsAgree:cor} applies to $\AA$ and
            shows that $A = \lim_{T} \AA$.
            Let $\pi_n\colon A\to A_n$ be the canonical maps and let $I_n =
            \ker(\pi_n)$. Let $\pi_n^M\colon M\to M_n$ be the canonical maps.
            The maps $\pi_n^M$ and $\pi_n$ are
            surjective by exactness of $(-)[\Chi]$.
            This
            implies that
            \[
                \pi_n^M(I_0^{n+1}M) = \pi_n(I_0^{n+1})M_n = \ker(A_n\to
                A_0)^{n+1}M_n = 0
            \]
            because $\AA$ and $\MM$ are adic. We have proven $\ker(\pi_n^M)
            \supset I_0^{n+1}M$. It remains to prove the other
            containment.
            Take an element $m\in \ker(\pi_n^M)$ and fix $\Chi$ such that $m\in
            \lim_{n}M_n[\Chi]$.
            By $T$-stabilization for the module $\MM$, for any ${n'} >
            n_{\Chi, \MM}$ the map
            $\lim_{n} M_n[\Chi] \to M_{n'}[\Chi]$ is an isomorphism. This
            isomorphism maps the element
            $m$ to an element
            of $\ker(M_{n'}\to M_n)= \ker(A_{n'}\to
            A_0)^{n+1}M_{n'}$. Since $\pi_{n'}$ and $\pi_{n'}^M$ are surjective, we have
            \[
                \pi_{n'}^M(I_0^{n+1}M) = \pi_{n'}(I_0)^{n+1}\pi_{n'}(M) = \ker(A_{n'}\to
                A_0)^{n+1}M_{n'}.
            \]
            Since $(-)[\Chi]$ is right-exact, we see that
            $\pi_{{n'}}^M((I_{0}^{n+1}M)[\Chi])$ contains $m$. So we pick an element
            $j\in (I_0^{n+1}M)[\Chi]$ such that
            $\pi_{n'}^M(j) = \pi_{n'}^M(i)$. Then $\pi_{n'}^M(i-j) = 0$ and
            by~\eqref{ref:ChiOnChiAlgebraization:eq} we have $i-j\in
            \lim_{n} M_n[\Chi]$. But $\pi_{n'}^M$ restricted to $\lim_{n}
            M_n[\Chi]$ is an isomorphism, so $i - j = 0$.
        \end{proof}
        \begin{proposition}[generating sets for grounded
            modules]\label{ref:presentationsBounded:prop}
            Let $\AA$, $\MM$, $A$, $M$ be as in
            Theorem~\ref{ref:existenceOfAlgebraization:thm}.
            Let $F$ be an $A$-module with a $T$-linearization and a $T$-equivariant
            homomorphism of $A$-modules $p\colon F\to M$
            such that the associated map $p\colon F_0 \to M_0$ is surjective.
            Then $p$ is surjective.
        \end{proposition}
        \begin{proof}
            Consider the map $p_n\colon F_n\to M_n$ which is just $p_n = p\tensor_A
            (A/I^{n+1})$. Let $G_n = p_n(F_n)$. Since $p_0$ is surjective and
            $\MM$ is adic, we have $M_n = G_n + IM_n$. But $I^{n+1}M_n = 0$ so
            \[
                M_n = G_n + IM_n= G_n + I(G_n + IM_n) = G_n + I^2M_n = \ldots
                = G_n + I^{n+1}M_n = G_n.
            \]
            Let $G = p(F) \subset M$. It is enough to prove that $G[\Chi] =
            M[\Chi]$ for all $\Chi$. By exactness, we have $M[\Chi] = \lim_n M_n[\Chi]$ for all $\Chi$.  By
            Stabilization Lemma~\ref{ref:stabilizationAlgebraically:lem}, the
            map $M[\Chi]\to M_n[\Chi]$ is an isomorphism for $n$ large enough
            and we obtain the following diagram
            \[
                \begin{tikzcd}
                    G[\Chi] \arrow[r, hook]\arrow[d, two heads] & M[\Chi]\arrow[d, "\simeq"]\\
                    G_n[\Chi] \arrow[r, "\simeq"] & M_n[\Chi]
                \end{tikzcd}
            \]
            so $G[\Chi]\to M[\Chi]$ is surjective as well as injective.
        \end{proof}
        \begin{corollary}[finitely generated
            modules]\label{ref:finitelygenerated:cor}
            Let $\AA$ be a standard adic formal $\Gbar$-algebra with an
            algebraization $(A, I)$ and let $\MM = (M_n)_n$ be an adic module over
            $\AA$ with a $T$-action.  Then, the following are equivalent
            \begin{enumerate}
                \item\label{it:fgone} the $A_0$-module $M_0$ is finitely generated,
                \item\label{it:fgtwo} every $A_n$-module $M_n$ is finitely generated,
                \item\label{it:fgthree} $\lim_{T}(\MM)$ is a finitely generated $A$-module.
            \end{enumerate}
            If these hold, we say that $\MM$ is a \emph{finitely generated}
            $\AA$-module.
        \end{corollary}
        \begin{proof}
            $\ref{it:fgthree}.\implies\ref{it:fgtwo}.$ Since $(-)[\Chi]$ is
            right-exact for every $\Chi$, the map $\lim_{T}(\MM)\to M_n$ is
            surjective so $M_n$ is a finitely generated $A$-module. But $\MM$ is
            adic, so $I^{n+1}$ annihilates $M_n$ and so $M_n$ is a finitely
            generated $A_n$-module.\\
            The implication $\ref{it:fgtwo}.\implies\ref{it:fgone}.$ is trivial
            and $\ref{it:fgone}.\implies \ref{it:fgthree}.$ follows from
            Proposition~\ref{ref:presentationsBounded:prop} since $\MM$ is
            grounded by assumption and $T$ is
            linearly reductive, so we can find a $T$-equivariant surjection
            from a free $A$-module $F$ with a $T$-linearization onto $M_0$ and its
            $T$-equivariant lift to $\lim_{T}(\MM)$.
        \end{proof}

        \newcommand{\Formalize}{\bd{Formalize}}%
        We are now ready to prove the main equivalence statement for grounded
        modules.
        Consider a standard adic formal $\Gbar$-algebra $\AA = (A_n)_n$ with
        algebraization $(A, I = \ker(A\to A_0))$ as in
        Theorem~\ref{ref:existenceOfAlgebraization:thm}, so that
        the algebra $\AA$ is a formalization of $A$. Consider the two
        functors
        \[
            \Alg := \lim_{\Group}(-)\colon \Mod_{\AA}^{\mathrm{adic}, \Group} \to
            \Mod_{A}^{\Group}\qquad\mbox{and}\qquad \Formalize\colon
            \Mod_A^{\Group}\to \Mod_{\AA}^{\mathrm{adic}, \Group},
        \]
        defined by $\Formalize(M) = (M/I^{n+1}M)_{n\in
            \mathbb{N}}$ and $\Alg(\MM) = \lim_{\Group}(\MM)$,
        where $\Mod_{A}^{\Group}$
        denotes the category of $A$-modules with a $\Group$-action and
        $\Mod_{\AA}^{\mathrm{adic}, \Group}$ of adic
        $\Group$-modules over $\AA$.
        \begin{theorem}[Algebraization for grounded $\Group$-modules]\label{ref:finitelyGeneratedAlgebraization:thm}
            The functors $\Alg$ and $\Formalize$ restrict to an equivalence
            between the category of finitely generated $A$-modules with a
            $\Group$-action and finitely generated adic $\AA$-modules (in the
            sense of Corollary~\ref{ref:finitelygenerated:cor}) with
            $\Group$-action.
        \end{theorem}

        \begin{proof}
            Using Corollary~\ref{ref:algebraizationsAgree:cor} we identity
            $\lambda$- and $\Chi$-algebraizations of finitely generated
            $\AA$-modules.
            By Corollary~\ref{ref:finitelygenerated:cor} the functors $\Alg$
            and $\Formalize$ map finitely generated modules to finitely
            generated ones. It remains to prove that they give an equivalence.

            By Theorem~\ref{ref:existenceOfAlgebraization:thm}, the
            composition $\Formalize \circ \Alg$
            is isomorphic to identity.
            Consider $\Alg \circ \Formalize$. Choose an $A$-module $M$ with
            a $\Group$-linearization, let $\MM = \Formalize(M)$ and $M' =
            \Alg(\MM)$ and consider the natural map $M\to M'$. By
            Stabilization Lemma~\ref{ref:stabilizationAlgebraically:lem}, the map
            $M[\Chi]\to M'[\Chi]$ is surjective for every $\Chi$. Since every element of $M'$
            sits inside some $M'[\Chi]$, we get that $M\to M'$ is surjective. By the same argument as in
            Stabilization Lemma, for every $\Chi$ the map $M[\Chi]\to (M/I^{n+1}M)[\Chi]$ is an
            isomorphism for $n$ large enough, so $M\to M'$ is injective.
        \end{proof}

    \section{Formal \BBname{} functors}

    \newcommand{\Zform}{\mathcal{Z}}%
    \newcommand{\II}{\mathcal{I}}%

        In this section we introduce the formal version of \BBname{} functors.
        The notation and general outline are consistent
        with~\cite[Section~6]{jelisiejew_sienkiewicz__BB}.
        For an $n\in \mathbb{Z}_{\geq 0}$ let $\Gbar_n = V(\mm_0^{n+1}) \subset \Gbar$ be the $n$-th thickening
        of the $\kk$-point $0\in \Gbar$. Consider the set-valued functor
        \[
            \Xhat(S) = \left\{ (\varphi_n)_{n\in \mathbb{Z}_{\geq 0}} \ |\
            \varphi_n \colon \Gbar_n \times S \to X,\ \varphi_n\mbox{ is
            }\Group\mbox{-equivariant and } (\varphi_{n+1})_{|\Gbar_n \times
            S} = \varphi_n\mbox{ for all }n \right\}.
        \]
        We have a natural \emph{formalization map} $\Xplus \to \Xhat$ given by
        $\varphi \mapsto (\varphi_{|\Gbar_n \times S})_n$. Eventually, we will
        prove that it is an isomorphism.
        The crucial technical advantage of $\Xhat$ over $\Xplus$ is that it
        avoids the topological
        issues altogether: the set-theoretic image of each $\varphi_n$ lies in
        $X^{\Group}$.

        \subsection{Formal $\Gbar$-schemes and $\Gbar$-linearized sheaves}
        We introduce some notation for the formal neighbourhoods of fixed
        points in $X$. It is convenient to make this in a very general
        way.
        \begin{definition}\label{def:formalschemes}
            A \emph{formal $\Group$-scheme}
            consists of a sequence of $\Group$-schemes $\Zform = (Z_{n})_{n\in
            \mathbb{Z}_{\geq 0}}$ and equivariant closed immersions
            \[
            \begin{tikzpicture}
                [description/.style={fill=white,inner sep=2pt}]
                \matrix (m) [matrix of math nodes, row sep=3em, column sep=3em,text height=1.5ex, text depth=0.25ex] 
                { Z_0 &  Z_1 &  \ldots  & Z_n & Z_{n+1} &  \ldots  \\} ;
                \path[right hook->,font=\scriptsize]
                (m-1-1) edge node[above] {$ $} (m-1-2)
                (m-1-2) edge node[above] {$ $} (m-1-3)
                (m-1-3) edge node[above] {$ $} (m-1-4)
                (m-1-4) edge node[above] {$ $} (m-1-5)
                (m-1-5) edge node[above] {$ $} (m-1-6);
            \end{tikzpicture}
            \]
            such that
            \begin{enumerate}
                \item the $\Group$-action on $Z_0$ is trivial and $Z_0 \into
                    Z_n^{\Group}$ is an isomorphism,
                \item if $\II \subset \OO_{Z_n}$ is the ideal sheaf defining
                    $Z_0 \subset Z_n$, then for all $m \leq n$ the ideal sheaf
                    $\II^{m+1}$ defines $Z_m \subset Z_n$.
            \end{enumerate}
            \emph{A morphism of formal $\Group$-schemes} $f\colon (W_n)\to (Z_n)$ is a family of
            $\Group$-equivariant morphisms $(f_n\colon W_n \to Z_n)_n$ compatible with
            closed immersions $W_{n}\into W_{n+1}$ and $Z_{n}\into
            Z_{n+1}$.
            If every $Z_n$ is a $\Gbar$-scheme and the closed immersions $Z_n\into
            Z_{n+1}$ are $\Gbar$-equivariant then we say that $\Zform$ is a
            \emph{formal $\Gbar$-scheme}. Similarly, a \emph{morphism of
            formal $\Gbar$-schemes} is a morphism
            of formal $\Gbar$-schemes such that every $f_n$ is
            $\Gbar$-equivariant.
            If every $Z_n$ is locally Noetherian, we say that $\Zform$ is
            \emph{locally Noetherian}.
        \end{definition}
        \begin{remark}\label{ref:formallocallylinear:rmk}
            For a formal $\Group$-scheme $\Zform = (Z_n)$ and any $m > n$ the
            radicals of ideal sheaves of $Z_n\subset Z_m$ and $Z_0\subset Z_m$ are equal, so
            $Z_0\into Z_n$ is a
            homeomorphism of topological spaces. Therefore for every $n$, the
            $\Group$-action on the topological space $|Z_n|$
            is trivial, so every $Z_n$ is locally
            $\Group$-linear.
        \end{remark}

        \begin{definition}[Quasi-coherent sheaves on $\Zform$]
            Let $\Zform = (Z_n)_n$ be a formal $\Group$-scheme. The closed
            inclusions $Z_n\into Z_{n+1}$ induce a sequence of restriction maps
            \begin{equation}\label{eq:telescopeOfQCoh}
                \begin{tikzcd}
                    \ldots \ar[r] & \Qcoh_{\Group}(Z_{n+1})\ar[r]  &  \Qcoh_{\Group}(Z_n)\ar[r] &
                     \ldots \ar[r] & \Qcoh_{\Group}(Z_0)
                \end{tikzcd}
            \end{equation}
            A \emph{quasi-coherent $\Group$-linearized sheaf} on $\Zform$ is a
            sequence $(F_n\in \QcohG{Z_n})_n$ together with isomorphisms
            $i_n\colon F_{n+1}|_{Z_n}\to F_n$.
            A \emph{morphism of quasi-coherent $\Group$-linearized sheaves}
            $\varphi\colon (F_{\bullet},i_{\bullet})\to
            (G_{\bullet},j_{\bullet})$ is a sequence of morphisms
            $\varphi_n\colon F_n\to G_n$ such that for every
            $n$ the following diagram commutes
            \[
                \begin{tikzcd}
                    F_{n+1}|_{Z_n} = F_{n+1}\tensor_{\OO_{Z_{n+1}}}
                    \OO_{Z_{n}} \arrow[d, "\varphi_{n+1}\tensor \id"]\arrow[r,
                    "i_n"] & F_n \arrow[d, "\varphi_n"]\\
                    G_{n+1}|_{Z_n} = G_{n+1}\tensor_{\OO_{Z_{n+1}}}
                    \OO_{Z_{n}}\arrow[r, "j_n"] & G_n
                \end{tikzcd}
            \]
            These objects and morphisms form the \textit{category
                $\QcohG{\Zform}$ of $\Group$-linearized quasi-coherent sheaves on
            $\Zform$} which is a symmetric monoidal category, that is a $2$-categorical limit of the
            diagram~\eqref{eq:telescopeOfQCoh}.
            We say that a $\Group$-linearized sheaf $F$ of $\Zform$ is
            \emph{of finite type} if there exists an affine open cover of
            $|Z_0|$ such
            that the restrictions of $F$ to each member of this open cover are
            finitely generated modules in the sense of
            Corollary~\ref{ref:finitelygenerated:cor}. We denote by
            $\QcohGft{\Zform}$ the full subcategory of $\QcohG{\Zform}$
            consisting of finite type sheaves.
        \end{definition}

        We have a natural formalization and --- much
        subtler --- algebraization in the $\Group$-equivariant setting, that
        are the geometric analogues of the ones from
        Section~\ref{sec:formalSchemes}.
        Eventually we will use them to show that the formalization map
        $\Xplus\to \Xhat$ is an isomorphism.

        \newcommand{\Zhat}{\hat{Z}}%
        \begin{definition}[Formalization]\label{ref:formalizationGeometric:def}
            Let $X$ be a $\Group$-scheme and $\II \subset \OO_X$ be the ideal
            defining $X^{\Group}$. The \emph{formal $\Group$-scheme $\Zhat$
            associated to $X$} is the sequence $(Z_n)_n$ where $Z_n =
            (V(\II^{n+1}))^{+}$ is the \BBname{} decomposition of the $(n+1)$-th
            thickening of the fixed point set; the decomposition exists because the action of
            $\Group$ on the topological space of
            $V(\II^{n+1})$ is trivial, so this scheme is locally linear and
            Proposition~\ref{ref:representabilityForLocLin:prop} applies. The inclusions $Z_n\to Z$ induce restrictions
            $\QcohG{Z}\to \QcohG{Z_n}$ which together give a natural
            \emph{comparison functor} $\QcohG{Z}\to
            \QcohG{\Zform}$ that is cocontinuous and tensor (see appendix
            for definitions) and
            preserves finite type sheaves.
        \end{definition}
        \begin{definition}[Algebraization]
            Let $\Zform$ be a formal $\Gbar$-scheme. An \emph{algebraization}
            of $\Zform$ is a $\Gbar$-scheme $Z$ such that
            $\Zform$ is isomorphic to the formal $\Group$-scheme
            associated to $Z$ and the associated restriction functor $\QcohGft{Z}\to
            \QcohGft{\Zform}$ is an equivalence.
        \end{definition}

        \subsection{Algebraization, geometric
        counterpart}\label{ssec:algebraization_geometric}

        To a formal $\Gbar$-scheme $\Zform =
        (Z_n)$ we may associate the
        family of maps $\pi_n\colon Z_n\to Z_0$ which are multiplications by
        $0\in \Gbar$. In particular, $\pi_{n+1}|_{Z_{n}} = \pi_n$ for all $n$.
        By Remark~\ref{ref:formallocallylinear:rmk} and
        Proposition~\ref{ref:representabilityForLocLin:prop}, the maps $\pi_n$ are
        affine, so we have a sequence of sheaves of quasi-coherent
        $\Gbar$-algebras on $Z_0$:
        \[
            A_{n} := (\pi_{n})_*\OO_{Z_n}
        \]
        and surjections $A_{n+1}\onto A_n$. We call $(A_n)$ the
        \emph{sheaf of $\Gbar$-algebras associated to $\Zform$}.
        The category $\RepGlambda$ behaves well when we
        pass from $\RepG$ to $\RepGbar$ and then to $\Gbar$-equivariant
        sheaves on a
        $\Gbar$-scheme $Z_0$ with a trivial $\Gbar$-action.
        Indeed, fix a quasi-coherent sheaf $A$ with a $\Gbar$-action and
        define $A[\lambda] \subset A$ by
        setting $H^0(U, A[\lambda]) := H^0(U, A)[\lambda]$ for every open $U$. Since
        $(-)[\lambda]$ is left exact, we obtain a subsheaf.
        For every open $U$ the algebra $H^0(U, \OO_{Z_0})$ is a trivial
        $\Gbar$-representation, so the structural map
        $H^0(U, \OO_{Z_0})\tensor_{\kk} H^0(U, A)\to H^0(U, A)$
        descends to
        \[
            H^0(U, \OO_{Z_0})\tensor_{\kk} H^0(U, A)[\lambda]\to
            H^0(U, A)[\lambda].
        \]
        and so $A[\lambda] \subset A$ is a subsheaf of $\Gbar$-modules and
        $\OO_{Z_0}$-modules; since the action of $\Gbar$ on $Z_0$ is trivial
        those structures commute.
        We can thus form $\lambda$-algebraizations of sheaves. For a sheaf of
        algebras $A$ its $\lambda$-algebraization is a sheaf of algebras as
        well.
        If we fix a central torus $T$ the same holds for $(-)[\Chi]$ instead
        of $(-)[\lambda]$.

        \begin{theorem}[algebraization of a formal
            $\Gbar$-scheme]\label{ref:algebraization:thm}
            Let $\Zform = (Z_n)$ be a locally Noetherian formal $\Gbar$-scheme with associated
            sheaf of algebras $(A_n)$.
            The scheme $Z = \Spec_{Z_0}(\lim_{\Group}(A_n))$ is the colimit of
            $Z_n$ in the category of locally linear $\Gbar$-schemes
            and it is an algebraization of $\Zform$. Moreover, the scheme $Z$ is locally
            linear and $\iinfty{Z}\colon Z\to Z^{\Group}$ is affine of finite
            type, so $Z$ is locally Noetherian.
        \end{theorem}

        \begin{proof}
            \def\KK{\mathcal{K}}%
            \def\JJ{\mathcal{J}}%
            Let $A = \lim_{\Group}(A_n)$.
            We also fix central Kempf torus $\Group_{m,\kkbar}\to Z(\Group)$
            and its image $T \subset Z(\Group)$. This is a one-dimensional
            torus, possibly non-split, see Subsection~\ref{ssec:KempfMonoids}.
            By Theorem~\ref{ref:existenceOfAlgebraization:thm} applied to
            sections of $A$, the map $A\to A_n$ is surjective and its
            kernel is equal to $\ker(A\to A_0)^{n+1}$ section-wise.
            Therefore indeed the formalization of $Z$ is $\Zform$.

            Now we prove that $Z = \colim Z_n$ in the category of locally
            linear $\Gbar$-schemes. Fix a locally linear $\Gbar$-scheme $W$
            with coherent maps $f_n\colon Z_n\to W$.
            In particular, we get a map $f_0\colon Z_0\to W^{\Group}$.

            By
            Proposition~\ref{ref:representabilityForLocLin:prop} the
            multiplication by $0_{\Gbar}$ map $\pi_W\colon W\to W^{\Group}$ is
            affine. Let $W' := W \times_{W^{\Group}} Z_0$. The map $W'\to Z_0$
            is affine and $\Gbar$-equivariant. We have induced maps $f'_n = f_n
            \times \pi_n\colon Z_n\to W'$ over $Z_0$.
            To sum up, we have
            the following situation and we want to find the dashed arrow
            \newcommand{\aff}{\mathrm{aff}}
            \[
                \begin{tikzcd}[row sep=large, column sep=large]
                    &Z\arrow[rd, dashed, bend left]&&\\
                    Z_n\arrow[ru, hook]\arrow[r, hook]\arrow[rrd,
                    "\pi_n"']\arrow[rr, "\phantom{mmmmm}f_n'",
                    bend left]
                    & Z_{n+1}\arrow[u, hook]\arrow[r,
                    "f_{n+1}'"]\arrow[rd, "\pi_{n+1}"]
                    & W' \arrow[r]\arrow[d,
                    "\aff"] & W\arrow[d, "\aff"]\\
                    && Z_0 \arrow[r, "f_0"] & W^{\Group}
                \end{tikzcd}
            \]
            All schemes $Z_n$, $Z$ and $W'$ are affine over $Z_0$. Let $B$
            be the sheaf of $\Gbar$-algebras on $Z_0$ such that
            $W'=\Spec_{Z_0}(B)$. Note that $B = \bigcup B[\lambda]$.
            The maps $f_n'$ induce compatible $\Gbar$-equivariant morphisms
            $f_n^{\#}\colon B\to A_n$. By $\Gbar$-equivariance, for every
            $\lambda$, they restrict to
            $f_{n}^{\#}\colon B[\lambda]\to A_n[\lambda]$ and give
            morphisms
            \[
                F[\lambda] : B[\lambda]\to A[\lambda].
            \]
            Composing those with $A[\lambda]\to A$, we obtain
            $F[\lambda]\colon B[\lambda]\to A$ which in turn glue to
            $F\colon B\to A$. (Compatibility of any two $F[\lambda_1]$ and
            $F[\lambda_2]$
            follows from compatibility of $f_n^{\#}$, because of stabilization.)
            This is a $\Gbar$-equivariant homomorphism of $\OO_{Z_0}$-algebras
            and it gives the required map $Z\to W'$ and in turn $Z\to W$.

            By assumption the scheme $\Zform$ is locally Noetherian; we will use
            it to prove that
            $Z\to Z_0$ is of finite type. The argument
            of~\cite[Theorem~6.8]{jelisiejew_sienkiewicz__BB} generalizes
            verbatim; we repeat it for completeness. By assumption the scheme $Z_0$ is
            locally Noetherian. For every $n$, in the $\OO_{Z_0}$-module $A_n$
            we have a \emph{coherent} ideal sheaf $\KK_n := \ker(A_n\to A_0)$ and by
            definition of formal schemes we have $\KK_n^n = \ker(A_n\to A_n) = 0$, so
            $0 = \KK_n^n \subset \KK_n^{n-1} \subset \ldots \subset \KK
            \subset A_n$ is a finite filtration whose quotients are
            coherent $\OO_{Z_n}$-modules that are in fact $\OO_{Z_0}$-modules.
            Thus the $\OO_{Z_0}$-module $A_n$ is coherent as well.
            Section-wise application Stabilization
            Lemma~\ref{ref:stabilizationAlgebraically:lem}
            implies that for every $\lambda$ the
            $\OO_{Z_0}$-module $A[\lambda]$ is coherent.
            The $A$-module $\JJ_0/\JJ_0^2 \subset A_2$ is coherent, so we
            can fix $\lambda$ and a map
            $A[\lambda]^{\oplus r}\to \JJ_0$ that restricts to a
            surjection $A[\lambda]^{\oplus r}\to \JJ_0/\JJ_0^2$.
            We claim that $A$ is generated as an $\OO_{Z_0}$-algebra
            by the image of $A[\lambda]^{\oplus r}$. To prove this, we pass
            to $\kkbar$ and use the grading induced by the Kempf torus. By
            Stabilization Lemma for the trivial
            representation, the degree zero part of $A$ is indeed $A_0 =
            \OO_{Z_0}$. Therefore $\JJ_0$ is positively graded and the result
            follows from graded Nakayama lemma.

            It remains to prove that $\QcohGft{Z}\to \QcohGft{\Zform}$ is an
            equivalence, but this follows from
            Theorem~\ref{ref:finitelyGeneratedAlgebraization:thm}.
        \end{proof}

        \subsection{Main proofs}

        In this short section we apply the results of the previous one to
        prove Theorem~\ref{ref:intro:mainthm}.

        \begin{proof}[Proof of Theorem~\ref{ref:intro:mainthm}]
        We begin by constructing a formal $\Gbar$-scheme and its
        algebraization.
        Let $X$ be a Noetherian $\Group$-scheme and $X_n = V(\II^{n+1})$ where $\II
        \subset \OO_X$ is the sheaf defining $X^{\Group}$, so $X_0 =
        X^{\Group}$.
        Let $Z_n = X_n^+$ be the \BBname{} decomposition, it is a closed
        subscheme of $X_n$, see Definition~\ref{ref:formalizationGeometric:def}.
        Let $\Zform =
        (Z_n)$ and let $\Xhat$ be the algebraization of $\Zform$ as in
        Theorem~\ref{ref:algebraization:thm}.
        The scheme $\Xhat$ has a natural map $\Xhat \to X^{\Group}$, but a
        priori lacks the ``inclusion of
        cells'' map $\Xhat \to X$.
        In fact, the construction of this inclusion map is subtle
        topologically, as $\Xhat$ is obtained from a formal neighbourhood of
        $X^{\Group}$ and does not see directly the topology of $X$.
        We now construct this map using the formalism of Tannaka duality,
        which is described in the
        appendix.

        By
        Theorem~\ref{ref:algebraization:thm}, we have $\QcohGft{\Xhat} =
        \lim_n \QcohGft{Z_n}$.
        The pullback via closed inclusions $Z_n \into X_n \into X$ gives a functor $F\colon
        \QcohGft{X}\to \lim_n \QcohGft{Z_n}$.
        The composition $F\colon \QcohGft{X}\to \QcohGft{\Xhat}$ is a cocontinuous
        tensor functor. Since $X$ is Noetherian, each $\Group$-linearized
        quasi-coherent sheaf on $X$ is a filtered limit of coherent
        $\Group$-linearized sheaves, see for
        example~\cite[Tag~07TU]{stacks_project}. The category of
        quasi-coherent sheaves of $\Xhat$ with $\Group$-linearization admits
        all limits, so the functor $F$ extends to a functor $\QcohGX\to
        \QcohG{\Xhat}$ that we also denote by $F$.
        Let $p_X\colon X\to \Spec(\kk)$ and $p_{\Xhat} \colon
        \Xhat\to \Spec(\kk)$ be the structure maps. Then $F\circ p_X^*$ and
        $p_{\Xhat}^*$ are isomorphic as this holds coherently for all $Z_n$.
        By Theorem~\ref{ref:tensorialquotientstacks:thm}, we obtain a
        $\Group$-equivariant map
        \[
            I\colon \Xhat\to X
        \]
        such that for every $n$ the map $I_{|Z_n}$ is the inclusion of $Z_n$.
            The $\Gbar$-scheme $\Xhat$ is a scheme over $X$ via the
            equivariant map $I$, so we have a family
            \[
                \Phi\colon \Gbar \times \Xhat \to X.
            \]
            It remains to prove that it is universal. For every other morphism
            $\varphi\colon \Gbar \times S\to X$ we have restrictions
            $\varphi_n\colon \Gbar_n \times S\to X_n$, which factor uniquely
            through $Z_n$ and hence give maps $S\to Z_n$. Since $\Xhat$ is a
            colimit of $Z_n$'s by
            Proposition~\ref{ref:algebraization:thm}, we obtain a map $S\to
            \Xhat$. Since $\varphi_n$ is a pullback of the canonical family on
            $Z_n$, the family $\varphi$ agrees to any finite order with the pullback of the
            universal family $\Phi$, hence $\varphi$ is equal to this pullback. This proves
            that $(\Xhat, \Phi)$ represents the functor $\Xplus$.
        \end{proof}

        Proposition~\ref{ref:intro:mainSmooth} now follows very similarly as the proof in the linearly
        reductive case~\cite[\S7]{jelisiejew_sienkiewicz__BB}. The key point
        is the mere existence and affineness of $\iinftyX\colon\Xplus\to X^{\Group}$.
        \begin{proof}[Proof of Proposition~\ref{ref:intro:mainSmooth}]
            Since $\iinftyX$ is affine and smooth at $x$ by assumption, we can
            write it locally on $X^{\Group}$ as $\pi\colon\Spec(B)\to
            \Spec(A)$, where $B$ is a $\Gbar$-algebra and $\pi$ is
            $\Gbar$-invariant. Since $\pi$ has a section, the map
            $\pi^{\#}\colon A\to B$ is injective.
            Suppose first $\kk = \kkbar$. Then a Kempf line sits inside
            $\Gbar$ and induces an $\mathbb{N}$-grading on $B$ such that $A
            \subset B_0$. By Corollary~\ref{ref:gradingPositive:cor}, we in fact have $A =
            B_0$. In this setup, the claim follows
            from~\cite[Lemma~7.2]{jelisiejew_sienkiewicz__BB}.

            For general $\kk$, we may assume $x$ is closed. Now the morphism $\Xplus_{\kkbar}\to
            (X^{\Group})_{\kkbar}$ is an affine space fibration near $x$ by
            the previous case, so
            in fact near $x$ it a trivial affine space fibration, i.e., a
            projection from a trivial vector bundle, so $\Xplus\to
            X^{\Group}$ is a $\GL$-torsor. But being a $\GL$-torsor
            bundle is Zariski-local, so we get that also $\iinftyX$ is such
            locally near $x$ and that concludes the proof.
        \end{proof}

\appendix
\section{Tannakian formalism}
    \newcommand{\divG}[1]{\left[ #1/\Group \right]}%
    \newcommand{\XdivG}{\divG{\varX}}%
    Let $\Group$ be a linear group over $\kk$ and $\varX$ be a
    $\Group$-scheme.
    In this appendix we recover $\Group$-equivariant
    morphisms from pullback-like maps. This is used
    crucially in the proof of Theorem~\ref{ref:intro:mainthm} to obtain the
    ``embedding of cells'' map $\Xhat \to \varX$.

    We need some preliminary notions.
    A \emph{cocontinuous} functor $F$ between tensor categories is a functor
    that preserves all small colimits; in
    particular it is right-exact. A \emph{tensor functor} is a strong symmetric
    monoidal functor, which means that $F(M\tensor N)  \simeq F(M)\tensor
    F(N)$ and $F(1)  \simeq 1$ and these isomorphisms are subject to
    compatibility
    conditions~\cite[I.4.1.1-4.2.4]{Rivano__Tannakian_cats}.
    For a $\Group$-scheme $T$ the equivariant map $p_{T}\colon T\to \Spec(\kk)$
    induces a pullback $p^{*}_{T}\colon \RepG\to \QcohG{T}$ which maps a
    $\Group$-representation $V$ to $V\tensor_{\kk} \OO_T$ with the natural
    linearization.
    For a $\Group$-equivariant morphism $\varphi\colon Y\to X$ of
    $\Group$-schemes, we have a natural isomorphism $\alpha_{\varphi}\colon
    \varphi^*\circ p_X^* \to p_Y^*$ of functors.
    Our main result is the following.

    \begin{theorem}\label{ref:tensorialquotientstacks:thm}
		Let $\varX$ be a quasi-compact quasi-separated $\Group$-scheme.
        Let $Y$ be a $\Group$-scheme. Let $F\colon \QcohG{X}\to
        \QcohG{Y}$ be a cocontinuous tensor functor and let $\alpha\colon
        F\circ p_X^*\to p_Y^*$ be an
        isomorphism of functors $\RepG \to \QcohG{Y}$.
        Then there exists a unique $\Group$-equivariant
        morphism $f\colon Y\to \varX$ such that $(F, \alpha)  \simeq (f^*,
        \alpha_f)$.
    \end{theorem}

    The most natural language of proving (and even formulating)
    Theorem~\ref{ref:tensorialquotientstacks:thm} is the theory of algebraic
    stacks. We recall below the bare minimum necessary for its proof. For an
    introduction to stacks, see~\cite{Olsson}.  For a sketch of the proof
    without using the language of stacks, see the first arXiv version of the
    present article.

    Let $\star = \Spec(\kk)$.
    Let $T$ be a $\Group$-scheme over $\Spec(\kk)$. The quotient stack $[T/\Group]$ is the category
    fibred in groupoids, whose $S$-points are pairs $(\pi, \gamma)$, where
    $\pi\colon P\to S$ is a principal $\Group$-bundle over $S$ and
    $\gamma\colon P\to T$ is a $\Group$-equivariant morphism.
    Since $\Group$ is affine, this stack has affine diagonal.
    There is a
    forgetful map $p_T\colon \divG{T}\to\divG{\star}$ that maps the pair
    $(\pi, \gamma)$ to the pair $(\pi, \pr\circ \gamma)$, where $\pr\colon
    T\to \Spec(\kk)$ is the structural map.
    The category $\QcohG{T}$ is equivalent to the category $\Qcoh(\divG{T})$
    of quasi-coherent sheaves on the stack $\divG{T}$,
    see~\cite[Example~9.1.19]{Olsson}.

    \newcommand{\Homqc}{\Hom_{r\tensor, \simeq}}%
    \newcommand{\HomqcoverBG}{\Hom_{r\tensor, \simeq,\divG{\star}}}%
    Let now $X$, $Y$ be two $\Group$-schemes over $\Spec(\kk)$.
    We now recall the definition of the groupoid $\Mor_{\divG{\star}}\left( \divG{Y}, \divG{X}
    \right)$. Its objects are pairs $(f, \alpha)$, where $f\colon \divG{Y}\to
    \divG{X}$ is a morphism of stacks and $\alpha\colon p_X\circ f\to p_Y$ is
    an isomorphism of functors. Its morphisms $(f, \alpha)\to (f', \alpha')$
    are isomorphisms of functors $\beta\colon f\to f'$ that satisfy
    $\alpha'\circ (p_X \beta) = \alpha$.
    We now define the groupoid
    $\HomqcoverBG(\QcohG{X}, \QcohG{Y})$. Its objects are pairs
    $(F, A)$ where $F\colon \QcohG{X}\to \QcohG{Y}$ is a cocontinuous tensor
    functor and $A\colon F\circ p_X^*\to p_Y^*$ is an isomorphism of functors.
    Its morphisms $(F, A)\to (F', A')$ are isomorphisms of functors $B\colon
    F\to F'$ that satisfy $A'\circ (B p_X^*) = A$. Both groupoids defined
    above are the usual mapping spaces of two categories fibered in groupoids
    over a third such category.
    Let $\Mor_{\Group}(Y, X)$ be the set of $\Group$-equivariant
    morphisms $\varphi\colon Y\to X$. There is a natural functor
    \[
        [-/\Group]\colon \Mor_{\Group}(Y, X)\to \Mor_{\divG{\star}}(\divG{Y},
        \divG{X})
    \]
    that maps $\varphi$ to $([\varphi/\Group], \id)$.
    As in~\cite{HR} there is also a natural functor
    \[
        \omega_{\divG{X}}(\divG{Y})\colon \Mor_{\divG{\star}}(\divG{Y},
        \divG{X})\to \HomqcoverBG(\QcohG{X}, \QcohG{Y}).
    \]
    Theorem~\ref{ref:tensorialquotientstacks:thm} can be restated as: the
    composition of the two functors above is an equivalence of categories.
    We will show that both functors are equivalences. For the first, this
    follows from general properties. For the second, this follows from
    certain results of~\cite{HR}, however not directly from their main
    theorem, which requires the stack $\divG{Y}$ to be locally excellent.

    \begin{lemma}
        The functor $[-/\Group]$ is an equivalence of categories.
    \end{lemma}
    \begin{proof}
        \def\strict{\mathrm{strict}}%
        Call a functor $(f, \alpha)\in \Mor_{\divG{\star}}(\divG{Y},
        \divG{X})$ \emph{strict} if $\alpha = \id$.
        We claim that for every $(f, \alpha)$ there is a uniquely defined
        strict functor $(f_{\strict}, \id)$ isomorphic to $(f, \alpha)$.
        Indeed, for a scheme $S$ and a bundle $\pi\colon P\to S$ with
        a map $\gamma\colon P\to Y$ we have $f((\pi, \gamma)) = (\pi',
        \delta)$ and $\alpha( (\pi, \gamma))$ is an isomorphism
        \[
            \begin{tikzcd}
                X & \ar[l, "\delta"] P'\ar[d, "\pi'"']\ar[r, "\alpha_{(\pi,
                \gamma)}"] & P\ar[ld, "\pi"]\\
                & S &
            \end{tikzcd}
        \]
        so we define $f_{\strict}( (\pi, \gamma))$ as the bundle $(\pi,
        \delta \circ \alpha^{-1}_{(\pi, \gamma)})$. This proves existence.
        Uniqueness follows directly the definition: the only isomorphism of strict
        functors is the identity.
        The above strictification shows that we only need to prove that
        $[-/\Group]$ induces a bijection onto the set of strict functors.
        Consider the trivial bundle $(\pr, \sigma_Y)$. Then $[\varphi/\Group](\pr,
        \sigma_{Y}) = (\pr, \tau)$ where $\tau(g, y) = g\varphi(y)$, in particular
        $\varphi = \tau(1_{\Group}, -)$. This shows that $[-/\Group]$ is injective.
        Note that we use strictness here, to have a canonical section
        $(1_{\Group}, -)$.
        It remains to show that $[-/\Group]$ is onto the set of strict
        functors. Consider a strict
        functor $(f, \id)$, let $(\pr, \tau) = f(\pr, \sigma_Y)$
        and define $\varphi := \tau(1_{\Group}, -)$.
        Consider a scheme $S$, a trivial bundle $\pr_2\colon \Group \times S\to
        S$ and an equivariant map $\gamma\colon \Group \times S\to Y$.
        We have commutative diagrams
        \[
            \begin{tikzcd}
                \Group\times \Group \times S\ar[d, "\pr_{23}"] \ar[r, "\mu\times\id_S"]& \Group
                \times S\ar[r, "\gamma"]\ar[d,"\pr_2"]& Y &
                \Group\times \Group \times S\ar[d, "\pr_{23}"] \ar[r,
                "\id_{\Group}\times \gamma"]& \Group
                \times Y\ar[r, "\sigma_Y"]\ar[d,"\pr_2"]& Y\\
                \Group \times S \ar[r, "\pr_2"]& S & &
                \Group \times S \ar[r, "\gamma"]& Y &
            \end{tikzcd}
        \]
        The bundles in the leftmost columns are equal and their maps to $Y$
        agree. Let $f(\pr_2, \gamma) = (\pr_2, \delta)$. Applying $f$ to both
        diagrams above, we get that $\delta(gh, s) = \tau(g, \gamma(h, s))$ for all
        $g,h\in \Group$ and $s\in S$. For $g = 1_{\Group}$ we
        get $\delta = \varphi\circ
        \gamma$ and
        thus $f$ maps $(\pr_2, \gamma)$
        to $(\pr_2, \varphi\circ \gamma)$. Applying this for
        $\gamma = \sigma_Y$ we deduce that $\varphi$ is equivariant. Therefore, the
        functors
        $f$ and $[\varphi/\Group]$ agree on all trivial bundles. By
        descent, they agree everywhere.
    \end{proof}

    \begin{proposition}[tensoriality]
        Suppose that $X$ is quasi-compact, quasi-separated. Then the functor
        $\omega_{\divG{X}}$ is an equivalence.
    \end{proposition}
    \begin{proof}
        The proof closely follows the argument in~\cite[Theorem~4.10]{HR}.
        The group $\Group$ is
        affine, so every quotient stack by $\Group$ has affine diagonal and
        then~\cite[Proposition~4.8(i)]{HR} implies that
        $\omega_{\divG{X}}(\divG{Y})$ and $\omega_{\divG{\star}}(\divG{Y})$ are fully faithful.
        To prove that $\omega_{\divG{X}}(\divG{Y})$ is
        essentially surjective, pick an element
        \[
            (F, A)\in \HomqcoverBG(\QcohG{X}, \QcohG{Y}).
        \]
        The group $\Group$ is smooth and we have the smooth affine atlas
        $s\colon X\to \divG{X}$. Let $\Lambda$ be the coordinate ring of
        $\Group$ with its natural linearization. The element
        $s_*\OO_X$ corresponds in $\QcohG{X}$ to the sheaf $\Lambda
        \tensor_{\kk} \OO_X \simeq p^*_X(\Lambda)$. Since $Fp_X^*$ is
        isomorphic to $p_Y^*$, the sheaf of algebras $F(s_*\OO_X)$ is isomorphic to
        $\Lambda \tensor_{\kk} \OO_Y\in \QcohG{Y}$.
        The stack $\Spec_{\divG{Y}}F(s_*\OO_X)$ is
        isomorphic to $\Spec_{\divG{Y}}(\Lambda\tensor_{\kk} \OO_Y) = Y$.
        By general nonsense and as explained in~\cite[Corollary~3.6]{HR} the functor $F$ induces a
        cocontinuous tensor functor
        \[
            F'\colon \Qcoh(X) \simeq \Qcoh(\Spec_{\divG{X}}(s_*\OO_X))
        \to \Qcoh(\Spec_{\divG{Y}}F(s_*\OO_X)) \simeq \Qcoh(Y).\]
        Now
        by~\cite{Brandenburg__Chirvasitu__tensoriality}, since $X$ is
        quasi-compact, quasi-separated, it follows that $F'$ is isomorphic to
        a pullback functor.
        The map $Y\to \divG{Y}$ is smooth, so
        by~\cite[Lemma~4.2(iii)]{HR} the functor $F$
        is isomorphic to a pullback functor of some $f\colon
        \divG{Y}\to\divG{X}$. The isomorphism $F \simeq f^*$ composed with $A$ gives an
        isomorphism
        $\alpha\colon f^*p_X^*\to p_Y^*$. The pair $(f, \alpha)$ is isomorphic to $(F, A)$, which
        concludes the proof.
    \end{proof}

\end{document}